\documentclass[10pt,a4paper]{article}
\usepackage[english]{babel}
\usepackage{polski}
\usepackage{amsmath, amsthm, amsfonts}
\usepackage{tikz-cd}
\usepackage{amssymb}

\usepackage{stackengine}
\usepackage{scalerel}

\usepackage{hyperref}
\usepackage{float}
\usepackage{subcaption}
\usepackage{wrapfig}

\usepackage{multicol}
\usepackage{amsmath}
\usepackage{amsfonts}
\usepackage{amsthm}
\usepackage{amssymb}
\usepackage{mathrsfs}
\usepackage{graphicx}
\usepackage{enumitem}
\usepackage{kbordermatrix}
\usepackage{tikz-cd}

\newcommand{\Spec}{\operatorname{Spec}}
\newcommand{\Ext}{\operatorname{Ext}}
\newcommand{\Proj}{\operatorname{Proj}}

\newcommand{\CC}{\mathbb{C}}
\newcommand{\ZZ}{\mathbb{Z}}
\newcommand{\PP}{\mathbb{P}}
\newcommand{\Sip}{\operatorname{Sip}}
\newcommand{\Slip}{\operatorname{Slip}}
\newcommand{\Hilb}{\operatorname{Hilb}}
\newcommand{\Hom}{\operatorname{Hom}}
\newcommand{\rrank}{\operatorname{r}}
\newcommand{\srr}{\operatorname{sr}}
\newcommand{\crr}{\operatorname{cr}}
\newcommand{\brr}{\operatorname{br}}
\newcommand{\Ann}{\operatorname{Ann}}

\newtheorem{theorem}{Theorem}[section]
\newtheorem{lemma}[theorem]{Lemma}

\newtheorem{problem}{Problem}
\newtheorem{proposition}[theorem]{Proposition}
\newtheorem{corollary}[theorem]{Corollary}
\newtheorem{remark}[theorem]{Remark}
\newtheorem{example}[theorem]{Example}

\theoremstyle{definition}

\title{Identifying limits of ideals of points in the case of projective space}
\begin{document}

\author{Tomasz Ma\'{n}dziuk\thanks{E-mail:~\textsf{t.mandziuk@mimuw.edu.pl}, Faculty of Mathematics, Computer Science and Mechanics, University of Warsaw, ul. Banacha 2, 02-097 Warszawa, Poland, ORCID: 0000-0002-3992-5890}}

\maketitle

\begin{abstract}
We study the closure of the locus of radical ideals in the multigraded Hilbert scheme associated with a standard graded polynomial ring and the Hilbert function of a homogeneous coordinate ring of points in general position in projective space. In the case of projective plane, we give a sufficient condition for an ideal to be in the closure of the locus of radical ideals. For projective space of arbitrary dimension we present a necessary condition.
The paper is motivated by the border apolarity lemma which connects such multigraded Hilbert schemes with the theory of ranks of polynomials.
\end{abstract}

\section{Introduction}\label{s:introduction}

Let $S^* = \CC[x_0,\ldots, x_n]$ be a polynomial ring and $F\in S^*_d$ for some positive integer $d$. The (Waring) rank of $F$ is $\min\{r\in \mathbb{Z} | F = L_1^d + \ldots + L_r^d \text{ for some } L_1,\ldots, L_r\in S^*_1\}$.
Finding the rank of a general homogeneous polynomial of $S^*$ of a given degree is a classical problem studied, among others, by Sylvester. For a given polynomial $F\in S^*_d$ and integer $r$, the condition that $F$ has Waring rank at most $r$ can be expressed algebraically by means of the classical apolarity lemma (see \cite[Thm.~5.3]{IK06}). We recall this here.
We may consider the polynomial ring $S=\CC[\alpha_0, \ldots, \alpha_n]$ acting on $S^*$ by partial differentiation. We denote this action by $\lrcorner$. Given $F\in S^*$ one may consider the annihilator ideal $\Ann(F)$ in $S$ consisting of all those polynomials $\theta$ for which $\theta \lrcorner F = 0$. Then the condition that $F$ has rank at most $r$ is equivalent to the existence of a homogeneous, radical ideal $I\subseteq \Ann(F)$ such that $S/I$ has Hilbert polynomial $r$.

Since, for a fixed integer $r$,  the locus of points in $\PP(S^*_d)$ corresponding to polynomials of rank at most~$r$ need not be closed, in the context of algebraic geometry, it is more natural to consider the Zariski closure of this locus. In that way, one obtains the secant variety $\sigma_r(\nu_d(\PP(S^*_1)))$. The polynomials in $S^*_d$ corresponding to points of $\sigma_r(\nu_d(\PP(S^*_1)))$ are said to have border rank at most $r$. 

In a recent paper \cite{buczyska2019apolarity}, a version of the apolarity lemma for the case of border rank was established. Its statement uses the notion of multigraded Hilbert scheme introduced in \cite{HS02}. For a fixed number of variables $n+1$ and positive integer $r$, one considers the multigraded Hilbert scheme $\Hilb_{S[\PP^n]}^{h_{r,n}}$ where $h_{r,n}$ is the generic Hilbert function of the homogeneous coordinate ring of $r$ points in $\PP^n$ (see below for a precise definition of $h_{r, n}$). This scheme has a distinguished irreducible component called $\Slip_{r, n}$ which is the closure of the locus of points corresponding to radical ideals. Then $F\in S^*_d$ has border rank at most~$r$ if and only if there is a point  $[I]\in \Slip_{r, n}$ such that $I \subseteq \Ann(F)$. The border apolarity lemma was shown to be an effective tool in the study of ranks of polynomials. See for instance \cite{CHL19}, \cite{GMR20} and \cite{HMV20} as a demonstration of its applicability.

As a consequence of the border apolarity lemma, the understanding of the closure involved in the definition of the secant varieties can be shifted to the problem of understanding the closure of the locus of radical ideals in the corresponding multigraded Hilbert scheme. This paper is concerned with finding conditions (necessary or sufficient) for a point $[I]\in \Hilb_{S[\PP^n]}^{h_{r, n}}$ to be in $\Slip_{r, n}$. In fact, the theory works in greater generality. One may replace $\CC$ by an algebraically closed field of arbitrary characteristic and replace $\PP^n$ by any smooth, projective toric variety $X$. Then the homogeneous coordinate ring $S$ of $\PP^n$ is replaced by the more general notion of a Cox ring of $X$ (see \cite{Cox95}). In this paper we restrict our attention to the case of projective space, but we intend to investigate analogous problems in greater generality in the future. 

Let $\Bbbk$ be a fixed algebraically closed field. In what follows we work in the categories of $\Bbbk$-algebras and $\Bbbk$-schemes. Given a point $x$ of a quasi-projective $\Bbbk$-scheme, we say that $x$ is closed if it is closed in the Zariski topology. This is equivalent to the condition that the residue field $\kappa(x)$ is the base field $\Bbbk$ (see \cite[Cor.~3.36]{GW10}). Let $n$ be a positive integer and $S=S[\mathbb{P}^n]$ be the homogeneous coordinate ring of $\mathbb{P}^n$, i.e. $S=\Bbbk[\alpha_0,\ldots, \alpha_n]$ is a standard $\mathbb{Z}$-graded polynomial ring. Fix also a positive integer $r$. 
Let $h\colon \mathbb{Z} \to \mathbb{Z}$ be the Hilbert function of $S[\mathbb{P}^n]/I$ where $I$ is a homogeneous saturated ideal of $S$ defining a zero-dimensional closed subscheme of $\mathbb{P}^n$ of length $r$. 
The multigraded Hilbert scheme corresponding to $S[\PP^n]$ and $h$ will be denoted by $\Hilb_{S[\mathbb{P}^n]}^h$. Let $\Sip_{h, n}$ denote the subset of closed points of $\Hilb_{S[\PP^n]}^{h}$ corresponding to saturated ideals of $r$-tuples of distinct points in $\mathbb{P}^n$. Let $\Slip_{h, n}$ be the closure of $\Sip_{h, n}$ in $\Hilb_{S[\PP^n]}^{h}$. 
Of special interest is the function  $h_{r, n}\colon \mathbb{Z}\to\mathbb{Z}$ given by $h_{r, n}(a) = \min \{\dim_\Bbbk S[\PP^n]_a, r\}.$ Observe that $h_{r, n}$ is the Hilbert function of the quotient $S[\PP^n]/I$ where $I$ is a saturated, homogeneous ideal of $r$-points in $\PP^n$ in general position. In that case, $\Slip_{r, n}:=\Slip_{h_{r, n}, n}$ is an irreducible component of $\Hilb_{S[\PP^n]}^{h_{r, n}}$. See \cite{buczyska2019apolarity} for the proof of this claim and more about these loci of $\Hilb_{S[\PP^n]}^{h_{r, n}}$.
We always consider $\Slip_{r,n}$ and other irreducible components of $\Hilb_{S[\PP^n]}^{h_{r,n}}$ with reduced scheme structure.
We will denote by $\mathfrak{m}$ the irrelevant ideal $(\alpha_0, \ldots, \alpha_n)$ of $S[\PP^n]$. Given a homogeneous ideal $\mathfrak{a}\subseteq S[\PP^n]$ we denote by $\sqrt{\mathfrak{a}}$ its radical and by $\overline{\mathfrak{a}}$ its saturation with respect to $\mathfrak{m}$.

After introducing the above general notation we may state the main results of the  paper. In Section~\ref{s:necessary_condition} we present a necessary condition for a closed point $[I]\in \Hilb_{S[\PP^n]}^{h_{r, n}}$ to be in $\Slip_{r, n}$. In fact, the same argument applies more generally in the case of $\Hilb_{S[\PP^n]}^{h}$ for any function $h\colon \mathbb{Z} \to \mathbb{Z}$ that is the Hilbert function of $S[\PP^n]/I$, where $I$ is a saturated homogeneous ideal defining a zero-dimensional closed subscheme of $\mathbb{P}^n$. 

\begin{theorem}\label{t:hilbert_function_of_square_criterion}
Let $n,r\geq 1$ be integers and $h\colon \mathbb{Z} \to \mathbb{Z}$ be the Hilbert function of $S[\PP^n]/I$ where $I$ is a saturated, homogeneous ideal defining a zero-dimensional, length $r$ closed subscheme of $\mathbb{P}^n$. Define $e=\min \{a\in \mathbb{Z} | h(a) = r\}$
and let $[I_0] \in \Hilb_{S[\PP^n]}^{h}$ be a closed point. If $[I_0] \in \Slip_{h,n}$, then $H_{S[\PP^n]/I_0^k}(d)\geq r\cdot \dim_\Bbbk S[\PP^n]_{k-1}$ for every positive integer $k$ and for every $d\geq ke+k$.
\end{theorem}

Based on Theorem \ref{t:hilbert_function_of_square_criterion}, we obtain in Theorem \ref{t:criterion_for_points_on_line} a necessary condition for an ideal $[I] \in \Hilb_{S[\mathbb{P}^n]}^{h_{r, n}}$ defining a subscheme of $\PP^n$ contained in a line to be in $\Slip_{r, n}$. As a corollary, we establish the following proposition concerning the minimal reducible example of $\Hilb_{S[\PP^n]}^{h_{r,n}}$.

\begin{proposition}\label{p:minimal_reducible}
For any positive integer $r$, the scheme $\Hilb_{S[\PP^1]}^{h_{r, 1}}$ is isomorphic to $\PP^r$ and in particular is irreducible. For $r\leq 3$, $\Hilb_{S[\PP^2]}^{h_{r, 2}}$ is irreducible. For $r=4$, $\Hilb_{S[\PP^2]}^{h_{4, 2}}$ is reducible. In fact, $[I] \notin \Slip_{4, 2}$ where $I = (\alpha_0\alpha_1, \alpha_0\alpha_2, \alpha_0^3, \alpha_1^4)$.
\end{proposition}

Proposition \ref{p:minimal_reducible} suggests that pathologies of the multigraded Hilbert schemes $\Hilb_{S[\PP^n]}^{h_{r,n}}$ occur for much simpler examples than for the corresponding Hilbert schemes $\mathcal{H}ilb_{r}(\PP^n)$. Therefore, it is expected that these pathologies might be studied more easily in the case of multigraded Hilbert schemes.

Theorem \ref{t:criterion_for_points_on_line} is then applied in Section \ref{s:applications} to the study of polynomials with small border rank. It seems that the criterion of Theorem \ref{t:criterion_for_points_on_line} can be substantially generalized. This is a work in progress by the author.

In Section \ref{s:sufficient}, for a fixed positive integer $r$, we consider functions $f\colon \mathbb{Z} \to \mathbb{Z}$ satisfying the condition:
\begin{equation}\label{eq:condition_star}
\text{there exist } e,d\in \mathbb{Z}_{>0} \text{ such that } f(a) = \begin{cases}
h_{r, 2}(a) & \text{ if } a\neq e\\
d & \text{ if }a=e.
\end{cases}
\tag{$\star$}
\end{equation}
We prove the following sufficient condition for a closed point $[I]\in \Hilb_{S[\PP^2]}^{h_{r, 2}}$ to be in $\Slip_{r, 2}$.

\begin{theorem}\label{t:almost_general_implies_lip}
Let $r$ be a positive integer and consider a closed point $[I_0]$ of the multigraded Hilbert scheme $\Hilb_{S[\PP^2]}^{h_{r, 2}}$. If the Hilbert function of $S[\PP^2]/\overline{I_0}$ satisfies \eqref{eq:condition_star}, then $[I_0]\in \Slip_{r, 2}$.
\end{theorem}

We also comment in Remarks \ref{r:generalization_1} and \ref{r:generalization_2} that natural generalizations of this result to the case of arbitrary dimension $n$ are not true.

\section{A necessary condition in the case of projective space}\label{s:necessary_condition}

We fix a positive integer $n$ and write $S$ for $S[\mathbb{P}^n] = \Bbbk[\alpha_0,\ldots, \alpha_n]$. Recall from Section \ref{s:introduction} that we denote the irrelevant ideal $(\alpha_0, \ldots, \alpha_n)\subseteq S$ by $\mathfrak{m}$. Moreover, given a homogeneous ideal $\mathfrak{a}\subseteq S$ we denote by $\sqrt{\mathfrak{a}}$ its radical and by $\overline{\mathfrak{a}}$ its saturation with respect to $\mathfrak{m}$.

\subsection{Algebraic results}
In this subsection we present some algebraic results concerning homogeneous ideals that will be used in the proof of Theorem \ref{t:hilbert_function_of_square_criterion}. 

Given a homogeneous ideal $\mathfrak{a}$ in $S$ we will consider the condition:
\begin{equation}\label{eqn:condition_asterisk}
\text{there exists a positive integer } d \text{ such that } S_d\subseteq \mathfrak{a} \tag{$\ast$}.
\end{equation}

The following lemma gives some simple properties of condition \eqref{eqn:condition_asterisk}.

\begin{lemma}\label{l:condition_asterisk}
Let $m\geq 2$ be an integer and $\mathfrak{a}$, $\mathfrak{b}$, $\mathfrak{a}_1, \ldots, \mathfrak{a}_m$ be homogeneous ideals of $S$. Then:
\begin{itemize}\itemsep0em
\item[(i)] $\overline{\mathfrak{a}\cap \mathfrak{b}} = \overline{\mathfrak{a}}\cap \overline{\mathfrak{b}}$;
\item[(ii)] $\mathfrak{a}+\mathfrak{b}$ satisfies \eqref{eqn:condition_asterisk} if and only if $\sqrt{\mathfrak{a}}+\sqrt{\mathfrak{b}}$ satisfies \eqref{eqn:condition_asterisk};
\item[(iii)] If $\mathfrak{a}_i+\mathfrak{a}_m$ satisfies \eqref{eqn:condition_asterisk} for all $i=1,\ldots, m-1$, then $\mathfrak{a}_1\cdot \ldots \cdot \mathfrak{a}_{m-1}+\mathfrak{a}_m$ and $\mathfrak{a}_1\cap \ldots \cap \mathfrak{a}_{m-1} + \mathfrak{a}_m$
satisfy \eqref{eqn:condition_asterisk};
\item[(iv)] If $\mathfrak{a}_1,\ldots, \mathfrak{a}_m$ are homogeneous ideals such that  $\mathfrak{a}_i+\mathfrak{a}_j$ satisfies \eqref{eqn:condition_asterisk} for all $1\leq i < j\leq m$, then $\overline{\mathfrak{a}_1\cdot\ldots\cdot \mathfrak{a}_m} = \overline{\mathfrak{a}_1\cap \ldots \cap \mathfrak{a}_m}$.
\end{itemize}
\end{lemma}
\begin{proof}
\begin{itemize}\itemsep0em
\item[(i)] Let $f\in \overline{\mathfrak{a}}\cap \overline{\mathfrak{b}}$. Then there are integers $k_1,k_2$ such that for all $i=0,\ldots, n$ we have $\alpha_i^{k_1}f\in \mathfrak{a}$ and $\alpha_i^{k_2}f \in \mathfrak{b}$. Therefore, $\alpha_i^{k_1+k_2}f\in \mathfrak{a}\cap \mathfrak{b}$ so that $\overline{\mathfrak{a}}\cap \overline{\mathfrak{b}}\subseteq \overline{\mathfrak{a}\cap \mathfrak{b}}$. The opposite inclusion follows from the fact that $\mathfrak{a}\cap \mathfrak{b}$ is contained in both $\mathfrak{a}$ and $\mathfrak{b}$ and thus, its saturation is contained in both $\overline{\mathfrak{a}}$ and $\overline{\mathfrak{b}}$.
\item[(ii)] It follows from \cite[Ex.~I.2.2]{H77} and \cite[Ex.~1.13~v)]{atiyah1994introduction}.
\item[(iii)] By part (ii) and \cite[Ex.~1.13~iii)]{atiyah1994introduction}, it is enough to check that 
$
\sqrt{\mathfrak{a}_1}\cap \ldots \cap \sqrt{\mathfrak{a}_{m-1}} + \sqrt{\mathfrak{a}_m}
$
satisfies~\eqref{eqn:condition_asterisk}. By assumptions and part (ii), there is an integer $d$ such that for all $i=0,\ldots, n$ and for all $j=1,\ldots,m-1$ there are elements $s_{ij}\in \sqrt{\mathfrak{a}_j}$ and $t_{ij}\in \sqrt{\mathfrak{a}_m}$ satisfying $\alpha_i^d = s_{ij} + t_{ij}$.
Therefore,
\[
\alpha_i^{d(m-1)} = \prod_{j=1}^{m-1} (s_{ij}+t_{ij}) = \prod_{j=1}^{m-1} s_{ij} + \Bigg(\prod_{j=1}^{m-1}(s_{ij}+t_{ij}) - \prod_{j=1}^{m-1}s_{ij}\Bigg).
\]
We have 
\[
\prod_{j=1}^{m-1} s_{ij}\in \sqrt{\mathfrak{a}_1}\cdot \ldots \cdot \sqrt{\mathfrak{a}_{m-1}}\subseteq \sqrt{\mathfrak{a}_1}\cap \ldots \cap \sqrt{\mathfrak{a}_{m-1}} 
\]
and  $\prod_{j=1}^{m-1}(s_{ij}+t_{ij}) - \prod_{j=1}^{m-1}s_{ij} \in \sqrt{\mathfrak{a}_m}$.
Hence
\[
\alpha_i\in \sqrt{\sqrt{\mathfrak{a}_1}\cap \ldots \cap \sqrt{\mathfrak{a}_{m-1}} + \sqrt{\mathfrak{a}_m}}
\]
for all $i=0,\ldots, n$ and thus, $\sqrt{\mathfrak{a}_1}\cap \ldots \cap \sqrt{\mathfrak{a}_{m-1}} + \sqrt{\mathfrak{a}_m}$ satisfies \eqref{eqn:condition_asterisk} by \cite[Ex.~I.2.2]{H77}.
\item[(iv)]
We prove it by induction on $m$ starting with $m=2$. Since $\mathfrak{a}_1\cdot \mathfrak{a}_2 \subseteq \mathfrak{a}_1\cap \mathfrak{a}_2$, it follows that $\overline{\mathfrak{a}_1\cdot\mathfrak{a}_2}\subseteq \overline{\mathfrak{a}_1\cap \mathfrak{a}_2}$. Moreover $(\mathfrak{a}_1+\mathfrak{a}_2)(\mathfrak{a}_1\cap \mathfrak{a}_2) \subseteq \mathfrak{a}_1\cdot \mathfrak{a}_2$. Since $\mathfrak{a}_1+\mathfrak{a}_2$ satisfies \eqref{eqn:condition_asterisk}, we obtain the inclusion $\mathfrak{a}_1\cap \mathfrak{a}_2\subseteq \overline{\mathfrak{a}_1\cdot \mathfrak{a}_2}$ and therefore, $\overline{\mathfrak{a}_1\cdot \mathfrak{a}_2} = \overline{\mathfrak{a}_1\cap\mathfrak{a}_2}$.

Let $m\geq 3$ and assume that (iv) is established for all integers smaller than $m$. By part (i) we have $
\overline{\mathfrak{a}_1\cap \ldots \cap \mathfrak{a}_m} = \overline{\mathfrak{a}_1\cap\ldots\cap \mathfrak{a}_{m-1}}\cap \overline{\mathfrak{a}_m}$.

Applying the inductive hypothesis for $m-1$ this can be rewritten as
\[
\overline{\mathfrak{a}_1\cap \ldots \cap \mathfrak{a}_m} =\overline{\mathfrak{a}_1\cap\ldots\cap \mathfrak{a}_{m-1}}\cap \overline{\mathfrak{a}_m}= \overline{\mathfrak{a}_1\cdot \ldots \cdot \mathfrak{a}_{m-1}}\cap \overline{\mathfrak{a}_m}.
\]
It follows from part (iii) that $\mathfrak{a}_1\cdot\ldots \cdot \mathfrak{a}_{m-1} + \mathfrak{a}_m$ satisfies \eqref{eqn:condition_asterisk}. Therefore, from part (i) and inductive hypothesis for $m=2$, we obtain
\[
\overline{\mathfrak{a}_1\cap \ldots \cap \mathfrak{a}_m} = \overline{\mathfrak{a}_1\cdot \ldots \cdot \mathfrak{a}_{m-1}}\cap \overline{\mathfrak{a}_m}=\overline{(\mathfrak{a}_1\cdot \ldots \cdot \mathfrak{a}_{m-1})\cap \mathfrak{a}_m} = \overline{\mathfrak{a}_1\cdot \ldots \cdot \mathfrak{a}_{m}},
\]
as claimed. 
\end{itemize}
\end{proof}

\begin{lemma}\label{l:saturation_of_power}
Let $I,J$ be homogeneous ideals of $S$ and $k$ be a positive integer. Then:
\begin{itemize}\itemsep0em
\item[(i)] There is an integer $d_0$ such that for all integers $d_1,\ldots, d_k\geq d_0$ the map $\bigotimes_{i=1}^k I_{d_i} \to I^k_{d_1+\ldots + d_k}$ induced by multiplication is surjective;
\item[(ii)] If $\overline{I} = \overline{J}$, then $\overline{I^k} = \overline{J^k}$.
\end{itemize}
\end{lemma}
\begin{proof}
\begin{itemize}\itemsep0em
\item[(i)] It is enough to consider a minimal set of homogeneous generators of $I$ and take $d_0$ to be the maximum of degrees of elements of this set, i.e. $d_0 = \max\{j \mid  \beta_{1,j}(S/I) \neq 0\}$, where $\beta_{i,j}$ are graded Betti numbers.
\item[(ii)] Let $d_0 = \max\{j \mid \beta_{1,j}(S/I)\neq 0\}$ and $e_0=\max\{j\mid \beta_{1,j}(S/J)\neq 0\}$. Let $r_0$ be an integer such that $I_{\geq r_0} = \overline{I}_{\geq r_0}$ and $J_{\geq r_0} = \overline{J}_{\geq r_0}$.  Let $s_0 = \max\{d_0, e_0, r_0\}$. Then for all $d_1,\ldots, d_k\geq s_0$ we have
\[
I^k_{d_1+\ldots +  d_k} = \overline{I}^k_{d_1+\ldots + d_k} = \overline{J}^k_{d_1 + \ldots + d_k} = J^k_{d_1+\ldots + d_k}
\]
where the first and last equality follow from part (i). In particular, $\overline{I^k} = \overline{J^k}$.
\end{itemize}
\end{proof}

After the above algebraic preparation we are able to compute the Hilbert polynomial of a power of a homogeneous radical ideal defining a closed, zero dimensional subscheme of projective space.

\begin{lemma}\label{l:hilbert_polynomial_of_power}
Let $I\subseteq S$ be a homogeneous radical ideal such that the Hilbert polynomial of the quotient algebra $S/I$ is constant, equal to $r$ for some positive integer $r$. Then for a positive integer $k$, the Hilbert polynomial of $S/I^k$ is constant equal to $r\cdot \dim_\Bbbk S_{k-1}$.
\end{lemma}
\begin{proof}
Let $\Proj S/I \subseteq \mathbb{P}^n$ be the set of (distinct) points $P_1, \ldots, P_r$ and define $\mathfrak{p}_i$ to be the prime ideal defining $P_i$. Then $I=\mathfrak{p}_1\cap \ldots \cap \mathfrak{p}_{r}$. Define $J=\mathfrak{p}_1\cdot \ldots \cdot \mathfrak{p}_r$. Since for all $1\leq i < j \leq m$ we have $\mathfrak{p}_i+\mathfrak{p}_j  = \mathfrak{m}$, it follows from Lemma \ref{l:condition_asterisk}(iv) that $\overline{I} = \overline{J}$. Hence $\overline{I^k} = \overline{J^k}$ by Lemma \ref{l:saturation_of_power}(ii). It is enough to show that the Hilbert polynomial of $S/J^k$ is $r\cdot \dim_\Bbbk S_{k-1}$. Let $K=\mathfrak{p}_1^k\cap \ldots \cap \mathfrak{p}_r^k$. 
Observe that $\mathfrak{p}_i^k + \mathfrak{p}_j^k$ satisfies \eqref{eqn:condition_asterisk} for every $1\leq i < j \leq m$ by Lemma \ref{l:condition_asterisk}(ii). Therefore, $\overline{K} = \overline{J^k}$ by Lemma \ref{l:condition_asterisk}(iv). Thus, it is enough to consider the Hilbert polynomial of $S/K$. Since, as a set, $\Proj S/K$ is the disjoint union of $r$-points $P_1,\ldots, P_r$ it is enough to show that the degree of $\Proj S/\mathfrak{p}_i^k$ is $\dim_\Bbbk S_{k-1}$ for every $i=1,\ldots, r$. This is clear, since up to a linear change of variables $\mathfrak{p}_i = (\alpha_1,\ldots, \alpha_n)$.
\end{proof}

\begin{remark}
In Lemma \ref{l:hilbert_polynomial_of_power} the assumption that $I$ is a radical ideal cannot be weakened to the assumption that $I$ is saturated. Indeed, $I=(\alpha_0^2, \alpha_0\alpha_1,\alpha_1^2) \subseteq S=\Bbbk[\alpha_0, \alpha_1,\alpha_2]$ is a saturated ideal and $S/I$ has Hilbert polynomial $3$, but $S/I^2$ has Hilbert polynomial $10$. 
\end{remark}

In Lemma \ref{l:hilbert_polynomial_of_power} we have calculated the Hilbert polynomial of $S/I^k$ for a homogeneous radical ideal $I$ defining a zero dimensional closed subscheme of a projective space and a positive integer $k$. The following proposition provides a bound on the least degree, from which the Hilbert function of $S/I^k$ agrees with the Hilbert polynomial of $S/I^k$.
Recall that for a finitely generated $\mathbb{Z}$-graded $S$-module $M$, its regularity $\operatorname{reg} M$ is defined to be $\operatorname{reg} M = \max\{j-i\mid \beta_{i,j}(M) \neq 0\}$.

\begin{proposition}\label{p:hilbert_function_of_power_in_sip}
Let $r,k$ be positive integers and $I\subseteq S$ be a homogeneous, radical ideal with the Hilbert polynomial of the quotient algebra $S/I$ equal to $r$. Define $e= \min\{a\in \mathbb{Z} \mid H_{S/I}(a) = r\}$. Then $H_{S/I^k}(d) = r\cdot \dim_\Bbbk S_{k-1}$ for $d\geq ke+k$.
\end{proposition}
\begin{proof}
By Lemma \ref{l:hilbert_polynomial_of_power}, the Hilbert polynomial of $S/I^k$ is $r\cdot \dim_\Bbbk S_{k-1}$.  Hence, by \cite[Thm.~4.2]{eisenbud2005geometry}, it is enough to show that $ke+k-1\geq \operatorname{reg} S/I^k$. 
Since $\operatorname{reg} S/I = e$ by \cite[Thm.~4.2]{eisenbud2005geometry}, it follows from the definition of regularity that $\operatorname{reg} I = e+1$. Thus, $\operatorname{reg}I^k\leq ke+k$ by \cite[Thm.~6]{Chandler-1997}. Therefore, $\operatorname{reg}S/I^k \leq ke+k-1$. 
\end{proof}

\subsection{Proof of Theorem \ref{t:hilbert_function_of_square_criterion}}

\begin{proof}[Proof of Theorem \ref{t:hilbert_function_of_square_criterion}]
Let $\mathscr{J}$ be the universal ideal sheaf on $\Hilb_{S}^{h} \times \mathbb{A}^{n+1}$. 
Consider the quotient $\mathscr{P}$ of $\mathcal{O}_{\Hilb_{S}^{h} \times \mathbb{A}^{n+1}}\cong \mathcal{O}_{\Hilb_{S}^{h}}[\alpha_0,\ldots,\alpha_n]$ by $\mathscr{J}^k$
 and let $\mathscr{Q}$ be the pushforward of $\mathscr{P}$ under the projection morphism
 $
 \pi\colon\Hilb_{S}^{h} \times \mathbb{A}^{n+1} \to \Hilb_{S}^{h}\text{.}
 $
 Then $\mathscr{Q} = \bigoplus_{d}\mathscr{Q}_d$ is a quasi-coherent sheaf on $\Hilb_{S}^{h}$ with $\mathscr{Q}_d$ a coherent sheaf for every $d\in \mathbb{Z}$. Therefore, for every $d\in \mathbb{Z}$, the rank function $\varphi_d\colon \Hilb_{S}^{h} \to \mathbb{Z}$ given by
 $
 \varphi_d(x) = \dim_{\kappa(x)} (\mathscr{Q}_d)_x \otimes_{\mathcal{O}_{\Hilb_{S}^{h},x}} \kappa(x) 
 $
 is upper-semicontinuous (see \cite[II,~Ex.~5.8]{H77}). We claim that for a closed point $P=[I]\in \Hilb_{S}^{h}$ we have $\varphi_d(P) = H_{S/I^k}(d)$.
 
This can be checked affine locally, so we can replace $\Hilb_{S}^{h}$ by an affine open subset $U=\Spec A$ containing $[I]$. Let $J$ be the ideal in $A[\alpha_0, \ldots, \alpha_n]$ defining the restriction of $\mathscr{J}$ to $\pi^{-1}(U)$. Let $[I]$ in $U$ correspond to the maximal ideal $\mathfrak{n}$ of $A$. In what follows, we will consider $\Bbbk$ with $A$-module structure given by $A\to A_\mathfrak{n}/\mathfrak{n}A_\mathfrak{n} \cong \Bbbk$. By the definition of universal ideal sheaf we have $(A[\alpha_0,\ldots, \alpha_n]/J)\otimes_A \Bbbk \cong S/I$. Therefore, from the universal property of kernel, there is an induced map $J\otimes_A \Bbbk \to I$ and it is surjective by snake lemma applied to the diagram
\[
\begin{tikzcd}
& J\otimes_A \Bbbk \arrow[d, dotted] \arrow[r] & A[\alpha_0, \ldots, \alpha_n]\otimes_A \Bbbk \arrow[d, equal] \arrow[r] & (A[\alpha_0, \ldots, \alpha_n]/J)\otimes_A \Bbbk) \arrow[d, "\cong"] \arrow[r] & 0 \\
0\arrow[r] & I\arrow[r] & S \arrow[r] & S/I \arrow[r] & 0.
\end{tikzcd}
\]
Hence also the map $J^k\otimes_A \Bbbk \to I^k$ is surjective. The snake lemma applied to the diagram
\[
\begin{tikzcd}
& J^k\otimes_A \Bbbk \arrow[d] \arrow[r] & S \arrow[d, equal] \arrow[r] & S/(J^k\otimes_A \Bbbk) \arrow[d, dotted] \arrow[r] & 0 \\
0\arrow[r] & I^k \arrow[r] & S \arrow[r] & S/I^k \arrow[r] & 0
\end{tikzcd}
\]
implies that the dotted arrow induced by the universal property of cokernel is injective. Since it is clearly surjective, it is an isomorphism.
Thus
\[
\varphi_d([I]) = \dim_\Bbbk \big(S/(J^k \otimes_A \Bbbk)\big)_d = \dim_\Bbbk (S/I^k)_d = H_{S/I^k}(d)
\]
and the claim of the theorem follows from Proposition \ref{p:hilbert_function_of_power_in_sip}.
\end{proof}

\subsection{Ideals of subschemes contained in a line}
In this subsection we prove Theorem \ref{t:criterion_for_points_on_line} which is a consequence of Theorem \ref{t:hilbert_function_of_square_criterion}. It provides a necessary condition for an ideal $[I]\in \Hilb_{S}^{h_{r, n}}$ defining a subscheme of $\mathbb{P}^n$ contained in a line to be in $\Slip_{r, n}$.

We start with the following scheme-theoretic result.
\begin{lemma}\label{l:components_of_intersection}
Let $X$ be a scheme locally of finite type over $\Bbbk$. Let $Z_1,Z_2$ be irreducible closed subsets of $X$ of dimensions $d_1,d_2$, respectively. Let $P\in Z_1\cap Z_2$ be a closed point of the intersection and let $d = \dim_\Bbbk T_P X$. Then every irreducible component $W$ of $Z_1\cap Z_2$ such that $P\in W$ satisfies $\dim W \geq d_1+d_2-d$.
\end{lemma}
\begin{proof}
By \cite[\href{https://stacks.math.columbia.edu/tag/0C2G}{Tag 0C2G}]{SP}, there exists an open neighborhood $U$ of $P$ in $X$ and a closed immersion $i\colon U \to Y$ where $Y$ is a smooth $d$-dimensional variety over $\Bbbk$. 
Let $W_1 = i(|Z_1|\cap |U|)$
 and $W_2 = i(|Z_2|\cap |U|)$, where 
 $| \cdot |$ denotes the underlying topological space. 
These are $d_1$ and $d_2$-dimensional irreducible closed subsets of $Y$, respectively. Therefore, every irreducible component of $W_1\cap W_2$ has dimension at least $d_1+d_2-d$ (see \cite[\S8.2]{F98}).
\end{proof}

Recall from Section \ref{s:introduction}, that $h_{r, n}$ is the Hilbert function of $S/I$ where $I$ is a saturated ideal of $r$ points of $\mathbb{P}^n$ in general position. Let $e=\min\{a\in \mathbb{Z} \mid h_{r, n}(a) = r\}$. We claim that under the assumptions $n\geq 2$, $r\geq 4$ we have

\begin{equation}\label{eq:estimation_in_terms_of_r}
2r-2 \geq 2e+2.
\end{equation}

Since $n\geq 2$ we get $\binom{2+e-1}{2}\leq \binom{n+e-1}{n}$. Therefore, since $\binom{n+e-1}{n}<r$, we obtain $2r>e(e+1)$ and thus, $2r\geq e(e+1)+2$. It follows that $2r-2\geq e^2+e$. If $e\geq 2$ then $e^2+e\geq 2e+2$ which gives the claimed relation $2r-2\geq 2e+2.$ If $e=1$ then $2r-2 \geq 2e+2$ since $r\geq 4$ by assumption.

Let $n\geq 2$ and $r\geq 4$ be integers. Consider the following condition on homogeneous ideals $I$ of $S$ with the Hilbert function of the quotient algebra $S/I$ equal to $h_{r, n}$:
\begin{equation}\label{eq:condition_dagger}
(\overline{I})^2\cdot \mathfrak{m}^{r-4}\subseteq I \tag{$\dagger$}.
\end{equation}
Observe that condition \eqref{eq:condition_dagger} is non-trivial only in degree $r-2$ since $\overline{I}_{ \geq r-1} = I_{ \geq r-1}$.

\begin{theorem}\label{t:criterion_for_points_on_line}
Let $n\geq 2$ and $r\geq 4$ be integers. Let $[I] \in \Hilb_{S}^{h_{r, n}}$ be such that $H_{S/\overline{I}}(1)=2$. If $[I]\in \Slip_{r, n}$ then $I$ satisfies \eqref{eq:condition_dagger}. Moreover, if $I$ satisfies \eqref{eq:condition_dagger}  and either $\operatorname{char} \Bbbk = 0$ or $\operatorname{char} \Bbbk \geq r-1$ then there exists a homogeneous ideal $J\subseteq S$ such that $[J]\in \Slip_{r, n}$ and $I_{\geq r-2} = J_{\geq r-2}$.
\end{theorem}
\begin{proof}
Up to a linear change of variables, $\overline{I} = (\alpha_0,\ldots, \alpha_{n-2}, F_r(\alpha_{n-1}, \alpha_n))$ where $F_r$ is non-zero and homogeneous of degree $r$. Suppose that $[I]\in \Slip_{r, n}$ but it does not satisfy \eqref{eq:condition_dagger}. 
Consider the monomial order $>_{\mathbf{w}, \operatorname{lex}}$ on $S$ given by the weight vector $\mathbf{w} = (1,1, \ldots, 1, 2, 2) \in \mathbb{Z}^{n+1}$ and a lexicographic monomial order $\operatorname{lex}$ (see \cite[Ex.~2.4.11]{CLO}).
Let $I_0$ be the initial ideal of $I$ with respect to $>_{\mathbf{w}, \operatorname{lex}}$.
 Then $[I_0]\in \Slip_{r, n}$ is a monomial ideal with $\overline{I_0} = (\alpha_0,\ldots, \alpha_{n-2}, F'_r(\alpha_{n-1}, \alpha_n))$ where $F'_r$ is a monomial of degree $r$. Moreover, $I_0$ does not satisfy \eqref{eq:condition_dagger}.

 We shall show that $H_{S/I_0^2}(2r-2) = r(n+1)-1$ which contradicts Theorem \ref{t:hilbert_function_of_square_criterion} and Equation \eqref{eq:estimation_in_terms_of_r}. Since $I_0^2$ is a monomial ideal, it follows from \cite[Thm.~15.3]{eisenbud1995commutative}
that it is enough to show that the number of monomials of degree $2r-2$ outside $I_0^2$ is $r(n+1)-1$.

Let $\mathfrak{a}=(\alpha_0,\ldots, \alpha_{n-2})$ and $\mathfrak{b}=(\alpha_{n-1}, \alpha_n)$. Since $H_{S/I_0}(r-1) = r = H_{S/\overline{I_0}}(r-1)$, it follows that $\mathfrak{a}\cdot\mathfrak{m}^{r-2}\subseteq I_0$. Hence $\mathfrak{a}^2\cdot \mathfrak{m}^{2r-4}\subseteq I_0^2$. 

Moreover, $\mathfrak{b}^{2r-2} \cap I_0^2 = 0$ since $\mathfrak{b}^{2r-2} \cap \overline{I_0}^2=0$. Therefore, $I_0^2$ contains no monomial of degree $2r-2$ in variables $\alpha_{n-1}, \alpha_n$. There are $2r-1$ such monomials. Thus, it is enough to show that there are exactly $r(n-1)$ monomials of the form $\alpha_iM$ where $i\in\{0,\ldots, {n-2}\}$, $M$ is a monomial of degree $2r-3$ in variables $\alpha_{n-1}, \alpha_n$ and $\alpha_iM \not\in I_0^2$.

Since $I_0$ does not satisfy \eqref{eq:condition_dagger}, it contains all monomials of the form $\alpha_iN$ where $i\in\{0,\ldots, n-2\}$ and $N$ is a monomial in variables $\alpha_{n-1}, \alpha_n$ of degree $r-3$. Therefore, $I_0^2$ does not contain a monomial of the form $\alpha_iM$ with $i\in\{0,\ldots,n-2\}$ and $M$ a monomial of degree $2r-3$ in variables $\alpha_{n-1}, \alpha_n$ if and only if $F'_r$ does not divide $M$. There are exactly $r(n-1)$ such monomials $\alpha_iM$.

Now we will prove the second part of the theorem.
Let $g\colon \mathbb{Z}\to \mathbb{Z}$ be defined by
\[
g(a) = \begin{cases}
\dim_\Bbbk S_a & \text{ if }a\leq r-3\\
h_{r, n}(a) & \text{ otherwise}.
\end{cases}
\]
Consider the multigraded Hilbert scheme $\Hilb_{S}^g$. We have morphisms
\[
\Hilb_{S}^{h_{r, n}} \xrightarrow{f_1} \Hilb_{S}^g \xrightarrow{f_2} \mathcal{H}ilb_r(\mathbb{P}^n) 
\]
where on closed points, $f_1$ maps $[\mathfrak{a}]$ to $[\mathfrak{a}\cap \mathfrak{m}^{r-2}]$ and $f_2$ maps $[\mathfrak{b}]$ to $[\Proj S/\mathfrak{b}]$.

Let $V\subseteq \mathcal{H}ilb_r(\mathbb{P}^n)$ be the closed subset defined by the set of its closed points
\[
V(\Bbbk) = \{[Z]\in \mathcal{H}ilb_r(\mathbb{P}^n) \mid H_{S/I_Z}(1)=2\}
\]
where $I_Z$ denotes the saturated ideal of $S$ defining the closed subscheme $Z\subseteq \Proj S$. Note that $V$ is irreducible of dimension $2(n-1)+r$. Let $W = f_2^{-1}(V)$ be the set-theoretic inverse image. Over every closed point $[Z]$ of $V$, the fiber of $f_2$ is the projective space $\PP^{\binom{n+r-2}{n} - (r-1)-1}$. Indeed, $H_{S/I_Z}(r-2) = r-1$ and $H_{S/I_Z}(a)=r$ for $a\geq r-1$. Therefore, we need to choose a codimension one subspace of $(I_Z)_{r-2}$. It follows that $W$ is irreducible of dimension $\binom{n+r-2}{n} + 2n - 2$.

Let $X$ be the closed subset of $W$, whose closed points satisfy \eqref{eq:condition_dagger}. This locus is irreducible of dimension $(2(n-1)+r) + (n-1)(r-2)-1 = nr-1$ since now the codimension one subspace of $(I_Z)_{r-2}$ as above must contain the subspace $((I_Z)_1)^2\cdot \mathfrak{m}_{r-4}$. We claim that $W \cap f_1(\Slip_{r, n}) = X$ set-theoretically. Note that this will finish the proof. By the first part of the proof, we know that $W \cap f_1(\Slip_{r, n}) \subseteq X$. Since $X$ is irreducible of dimension $\dim f_1(\Slip_{r, n}) -1$, it is enough to show that there is a point $P\in W \cap f_1(\Slip_{r, n})$ satisfying $\dim_\Bbbk T_P \Hilb_S^g = \dim W + 1 = \binom{n+r-2}{n} + 2n - 1$ (see Lemma \ref{l:components_of_intersection}).

We have $f_1^{-1}(f_2^{-1}(V))\cap \Slip_{r, n} \neq \emptyset$. Therefore, $W\cap f_1(\Slip_{r,n})\neq \emptyset$. Hence there is a Borel-fixed ideal $I_0\subseteq S$, such that $[I_0]$ is in $W\cap f_1(\Slip_{r,n})$. It follows from the assumption on the characteristic of $\Bbbk$ and \cite[Thm.~15.23]{eisenbud1995commutative} that $\overline{I_0} = (\alpha_0,\ldots, \alpha_{n-2}, \alpha_{n-1}^{r})$ and
\[
I_0=(\alpha_0^{r-2}, \alpha_0^{r-3}\alpha_1, \ldots, \alpha_{n-2}\alpha_{n-1}\alpha_n^{r-4}, \alpha_{n-2}\alpha_n^{r-2}, \alpha_{n-1}^r)
\]
where the generators of degree $r-2$ are all monomials divisible by one of the forms $\alpha_0,\ldots, \alpha_{n-2}$ except for $\alpha_{n-2}\alpha_n^{r-3}$. Therefore, to conclude the proof, it is enough to show that $\dim_\Bbbk \Hom_S(I_0, S/I_0)_0 \leq \binom{n+r-2}{n} + 2n - 1$. This is computed in Lemma \ref{l:calculation_of_tangent_space}.
\end{proof}

\begin{lemma}\label{l:calculation_of_tangent_space}
Let $I_0 = (\alpha_0^{r-2}, \alpha_0^{r-3}\alpha_1,\ldots, \alpha_{n-2}\alpha_{n-1}\alpha_n^{r-4}, \alpha_{n-2}\alpha_n^{r-2},\alpha_{n-1}^{r})$
where generators of degree $r-2$ are all monomials of degree $r-2$ in the ideal $(\alpha_0,\ldots, \alpha_{n-2})$ except for $\alpha_{n-2}\alpha_n^{r-3}$. Then 
\[
\dim_\Bbbk \Hom_S(I_0, S/I_0)_0 \leq \binom{n+r-2}{n} + 2n - 1.
\]
\end{lemma}
\begin{proof}
Let $\mathfrak{n} = (\alpha_0,\ldots, \alpha_{n-2})$ and pick $\phi\in \Hom_S(I_0, S/I_0)_0$.
For every positive integer $a$ we identify the $\Bbbk$-vector space $(S/I_0)_a$ with the $\Bbbk$-vector space spanned by monomials of degree $a$ in $S$ that are not in $I_0$. 

Let $i,j\in \{0,\ldots, n-2\}$ and consider a monomial $\alpha_i\alpha_jM \in I_0$ of degree $r-2$. Then also $\alpha_i\alpha_{n-1}M \in I_0$. Therefore, $\alpha_{n-1}\phi(\alpha_i\alpha_jM) = \alpha_j\phi(\alpha_i\alpha_{n-1}M) = 0$. It follows that on generators $\alpha_0^{r-2},\ldots, \alpha_{n-2}^2\alpha_n^{r-4}$ (all monomials of degree $r-2$ in $I_0$ that are in $\mathfrak{n}^2$) $\phi$ is given by the transpose of the matrix

\[
\kbordermatrix{ & {\alpha_{n-2}\alpha_n^{r-3}} & {\alpha_{n-1}^{r-2}} & {\alpha_{n-1}^{r-3}\alpha_n} & {\ldots} & {\alpha_n^{r-2}}\\
{\alpha_0^{r-2}} & {\ast} & 0 & 0 & 0 & 0\\
{\ldots} & {\ast} & 0 & 0 & 0 & 0\\
{\alpha_{n-2}^2\alpha_n^{r-4}} & {\ast} & 0 & 0 & 0 & 0
},
\]
i.e. for every monomial $N$ in $I_0\cap \mathfrak{n}^2$ of degree $r-2$ we can choose only the coefficient of $\phi(N)$ at $\alpha_{n-2}\alpha_n^{r-3}$.

Next, fix $i\in \{0,\ldots, n-3\}$ and consider the following generators of $I_0$: 
\[
M_{i,0}:=\alpha_i\alpha_{n-1}^{r-3}, M_{i,1}:=\alpha_i\alpha_{n-1}^{r-4}\alpha_n, \ldots, M_{i,r-3}:=\alpha_i\alpha_n^{r-3}.
\]
Using relations of the form $\alpha_nM_{i,s} = \alpha_{n-1}M_{i,s+1}$ for $s=0,\ldots, r-4$ we see that on generators 
\[
M_{i,0}, M_{i,1}, \ldots, M_{i,r-3},
\]
$\phi$ is given by the transpose of the matrix
\[
\kbordermatrix{& {\alpha_{n-2}\alpha_n^{r-3}} & {\alpha_{n-1}^{r-2}} & {\alpha_{n-1}^{r-3}\alpha_n} & {\alpha_{n-1}^{r-4}\alpha_n^2} & {\ldots} & {\alpha_{n-1}\alpha_n^{r-3}} & {\alpha_n^{r-2}}\\
M_{i,0} & \ast & A_i & B_i & 0 & 0& 0 & 0\\
M_{i,1} & \ast & 0 & A_i & B_i & 0& 0& 0\\
\ldots & \ast & 0  & 0 & \ldots & \ldots & \ldots  & 0 \\
M_{i, r-3} & \ast & 0 & 0 & 0& 0& A_i & B_i
}. 
\]
More precisely, there are $A_i,B_i \in \Bbbk$ such that 
\[
\phi(M_{i,j}) \equiv A_i \alpha_{n-1}^{r-2-j}\alpha_n^j + B_i \alpha_{n-1}^{r-3-j}\alpha_n^{j+1} (\operatorname{mod} \alpha_{n-2}\alpha_n^{r-3}).
\]
for all $j=0,\ldots, r-3$.

Similarly, considering the generators 
\[
M_{n-2,j}:=\alpha_{n-2}\alpha_{n-1}^{r-3-j}\alpha_n^j\text{ for }j=0,\ldots, r-4 \text{ and } \alpha_{n-2}\alpha_n^{r-2}
\]
we see that there are $A_{n-2}, B_{n-2} \in \Bbbk$ such that
\[
\phi(M_{n-2,j}) \equiv A_{n-2} \alpha_{n-1}^{r-2-j}\alpha_n^j + B_{n-2} \alpha_{n-1}^{r-3-j}\alpha_n^{j+1} (\operatorname{mod} \alpha_{n-2}\alpha_n^{r-3})
\]
for $j=0, \ldots, r-4$ and $\phi(\alpha_{n-2}\alpha_n^{r-2}) \equiv A_{n-2}\alpha_{n-1}\alpha_n^{r-2} + B_{n-2}\alpha_n^{r-1} (\operatorname{mod} \alpha_{n-1}^{r-1})$.

Together, these estimates show that $\dim_\Bbbk \Hom_S(I_0, S/I_0)_0 \leq \Big(\binom{n+r-2}{n} - r\Big) + 1 + r + 2(n-1)$ where:
\begin{itemize}\itemsep0em
\item[(1)] $\binom{n+r-2}{n} - r$ corresponds to the choices of coefficients at $\alpha_{n-2}\alpha_n^{r-3}$ for every monomial of degree $r-2$ in $I_0$;
\item[(2)] $1+r$ corresponds to the choice of coefficient at $\alpha_{n-1}^{r-1}$ for $\phi(\alpha_{n-2}\alpha_n^{r-2})$ and the choice of $\phi(\alpha_{n-1}^r)$;
\item[(3)] $2(n-1)$ corresponds to the choice of $A_0,B_0, A_1, \ldots , A_{n-2}, B_{n-2}$.
\end{itemize}
\end{proof}

We end this section with the proof of Proposition \ref{p:minimal_reducible} concerning the minimal example of a reducible multigraded Hilbert scheme $\Hilb_{S[\PP^n]}^{h_{r, n}}$.

\begin{proof}[Proof of Proposition \ref{p:minimal_reducible}]
For a positive integer $r$, we have $\operatorname{Hilb}_{S[\mathbb{P}^1]}^{h_{r,1}}\cong \mathcal{H}ilb_r(\PP^1)\cong \PP^r$ where the first isomorphism follows from \cite[Lem.~4.1]{HS02} (both schemes are the multigraded Hilbert scheme corresponding to the same Hilbert function) and the second isomorphism follows from \cite[pages 111-112]{FGI05}.

Now we consider the case of $\PP^2$. For $r\leq 3$, the scheme $\Hilb_{S[\PP^2]}^{h_{r, 2}}$ is irreducible by Theorem \ref{t:almost_general_implies_lip} and Lemma \ref{l:properties_of_hilbert_function_of_saturated_ideal}. Consider the point $[I] \in \Hilb_{S[\PP^2]}^{h_{4, 2}}$ where $I = (\alpha_0\alpha_1, \alpha_0\alpha_2, \alpha_0^3, \alpha_1^4)$. It follows from Theorem~\ref{t:criterion_for_points_on_line} that $[I]\notin \Slip_{4, 2}$ which implies that $\Hilb_{S[\PP^2]}^{h_{4, 2}}$ is reducible.
\end{proof}

\begin{remark}
Proposition \ref{p:minimal_reducible} shows that in \cite[Cor.~6.3]{buczyska2019apolarity} condition (iv) is not implied by conditions (i)-(iii). See the comment after \cite[Cor.~6.3]{buczyska2019apolarity}.
\end{remark}

\section{Applications to border ranks of polynomials}\label{s:applications}
In this section we assume that the base field $\Bbbk$ is the field of complex numbers $\CC$ since we will cite papers in which this is assumed. Let $n$ be a positive integer and $S=\CC[\alpha_0,\ldots, \alpha_n]$ be the polynomial ring with standard $\mathbb{Z}$-grading. We will consider the dual polynomial ring $S^* = \CC[x_0,\ldots,x_n]$ with the structure of an $S$-module on $S^*$ given by partial differentiation. We will denote this action by $\lrcorner$. Given a homogeneous polynomial $F \in S^*$ we denote by $\Ann(F)$ the ideal $\{\theta \in S | \theta \lrcorner F = 0\}$.

We recall various definitions of ranks of polynomials and state some versions of apolarity lemmas. See \cite{Tei14} for a discussion of related results and \cite{buczyska2019apolarity} for the case of border rank. 

Let $F \in S^*_d$. We define the rank of $F$ to be
\[
\rrank(F) = \min \{r\in \mathbb{Z} | F = L_1^d + \ldots + L_r^d \text{ for some } L_1, \ldots, L_r \in S^*_1\}.
\]
This can be generalized in various ways. We define the smoothable rank of $F$ to be
\begin{equation*}
\begin{split}
\srr(F) = \min \{r\in \mathbb{Z} | [F] \in \langle \nu_d(R) \rangle \text{ for  a smoothable zero-dimensional closed subscheme }\\
 R\subseteq \mathbb{P}(S^*_1) \text{ of length } r\},
\end{split}
\end{equation*}
where $\nu_d\colon \PP(S^*_1) \to \PP(S^*_d)$ is the Veronese map $[L]\mapsto [L^d]$ and $\langle \nu_d(R) \rangle$ is the projective linear span of the scheme $\nu_d(R)$. We define the cactus rank of $F$ to be
\begin{equation*}
\crr(F) = \min \{r\in \mathbb{Z} | [F] \in \langle \nu_d(R) \rangle \text{ for  a zero-dimensional closed subscheme }
 R\subseteq \mathbb{P}(S^*_1) \text{ of length } r\}.
\end{equation*}
It follows from these definitions, that $\crr(F) \leq \srr(F) \leq \rrank(F)$.
Recall that the $r$-th secant variety of the $d$-th Veronese variety is
$
\sigma_{r}(\nu_d(\PP (S^*_1))) = \overline{\{[F]\in \PP(S^*_d) | \rrank(F) \leq r\}}.
$
Finally, we define the border rank of $F$ to be
\[
\brr(F) = \min\{r\in \ZZ | [F] \in \sigma_r(\nu_d(\PP (S^*_1)))\}.
\]

We will use the following two apolarity lemmas.

\begin{proposition}[Cactus apolarity lemma, {\cite[Thm.~4.3]{Tei14}}]\label{p:cactus_apolarity} Let $F \in S^*_d$ be a non-zero polynomial and $r\in \mathbb{Z}_{>0}$. Then $\crr(F) \leq r$ if and only if there is a saturated, homogeneous ideal $I\subseteq S$ such that $I\subseteq \Ann(F)$ and $S/I$ has Hilbert polynomial $r$.
\end{proposition}

The version for border rank is a more recent result.

\begin{proposition}[Border apolarity lemma, {\cite[Thm.~3.15]{buczyska2019apolarity}}]\label{p:border_apolarity} Let $F \in S^*_d$ be a non-zero polynomial and $r\in \mathbb{Z}_{>0}$. Then $\brr(F) \leq r$ if and only if there is a homogeneous ideal $I\subseteq S$ such that $I\subseteq \Ann(F)$ and $[I]\in \Slip_{r, n}$.
\end{proposition}

We will use the following observation.

\begin{corollary}\label{c:bound_on_cactus_rank}
Let $F\in S^*_d$ and assume that there exists a homogeneous ideal $I\subseteq \Ann(F) \subseteq S$ such that  $S/I$ has Hilbert polynomial $r\in \ZZ_{>0}$ and $I_d = \overline{I}_d$. Then the cactus rank $\crr(F)$ of $F$ is at most $r$.
\end{corollary}
\begin{proof}
This is implied by Proposition \ref{p:cactus_apolarity} and \cite[Prop.~3.4]{BB14}.
\end{proof}

We always have $\brr(F) \leq \srr(F)$ and we say that $F$ is wild, if the inequality is strict. 
Wild polynomials are more difficult to control using standard, existing methods. Therefore, new methods need to be developed in order to study them effectively. For example, see \cite[Prop.~11]{BGI09} and its applications, \cite[Rmk.~1.5]{BL14} and \cite{GMR20}.

We will study wild homogeneous polynomials $F$ such that $\brr(F)\leq \deg(F)+2$.
If $\brr(F) \leq \deg(F)+1$ then $F$ is not wild. This is established in \cite[Prop.~2.5]{BB14} based on a result in \cite{BGI09}. Therefore, we will assume that $\brr(F) = \deg (F) + 2$.

\subsection{Polynomials in three variables of small border rank}
In this subsection we will prove the following proposition which shows that there are no wild homogeneous polynomials $F$ in three variables of border rank at most $\deg(F)+2$.
\begin{proposition}\label{p:wild_on_plane}
Let $S=\CC[\alpha_0,\alpha_1,\alpha_2]$ be a polynomial ring with dual ring $S^*=\CC[x_0,x_1,x_2]$. Let $F\in S^*_d$ be a non-zero polynomial for some $d\in \mathbb{Z}_{>0}$. If the border rank of $F$ is at most $d+2$, then $\crr(F)=\srr(F) = \brr(F)$.
\end{proposition}
\begin{proof}
By \cite[Prop.~2.5]{BB14} we may assume that $\brr(F) = d+2$. From Proposition \ref{p:border_apolarity} we obtain that there is an ideal $I\subseteq \Ann(F)$ such that $[I]\in \Slip_{d+2, 2}$. If $S/\overline{I}$ has a Hilbert function different from $h_{d+2, 1}$ then    $I_d = \overline{I}_d$ and therefore, $\crr(F) \leq d+2$ by Corollary \ref{c:bound_on_cactus_rank}. The Hilbert scheme $\mathcal{H}ilb_{d+2}(\PP^2)$ is irreducible, hence $\crr(F) = \srr(F)$. Moreover, the smoothable rank is bounded from below by the border rank. As $\brr(F) = d+2$ by assumption, we get that $\crr(F)=\srr(F) = d+2$.

Therefore, we may assume that $\Ann(F)$ contains an ideal $I$ such that $[I]\in \Slip_{d+2, 2}$ and 
\[
\overline{I} = (\alpha_0, F_{d+2}(\alpha_1,\alpha_2)).
\]
It follows from Theorem \ref{t:criterion_for_points_on_line} that $(\alpha_0^2)\cdot (\alpha_0, \alpha_1,\alpha_2)^{d-2} \subseteq I \subseteq \Ann(F)$. Thus, $\alpha_0^2 \in \Ann(F)$.

In particular, $F = x_0H_{d-1} + H_d$ for some homogeneous $H_{d-1}, H_d \in T^*:=\CC[x_1,x_2]$ of degrees $d-1, d$, respectively. It is enough to show that there is a homogeneous polynomial $\eta \in S$, not divisible by $\alpha_0$,  of degree $\lceil \frac{d+1}{2} \rceil$ such that $\eta \in \Ann(F)$. Indeed, then $(\alpha_0^2, \eta) \subseteq \Ann(F)$ is a saturated ideal of $S$ with Hilbert polynomial of the quotient algebra at most $d+2$ so from Proposition \ref{p:cactus_apolarity} we get that $\crr(F) \leq d+2$. Thus, from inequalities $d+2 = \brr(F) \leq \srr(F) =  \crr(F) \leq d+2$ we conclude that $\crr(F) = \srr(F) = d+2$.

Let $e=\lceil \frac{d+1}{2} \rceil$ and  $T=\CC[\alpha_1, \alpha_2]$. We will consider the restriction of the action of $S$ on $S^*$ to an action of $T$ on $T^*$. We shall show that there exist homogeneous polynomials $\xi_{e-1}, \xi_{e} \in T$ such that $\alpha_0\xi_{e-1} + \xi_{e} \in \Ann(F)$ and $\xi_{e} \neq 0$.
We have $(\alpha_0\xi_{e-1} + \xi_{e})\lrcorner F = x_0(\xi_{e}\lrcorner H_{d-1}) + (\xi_{e-1} \lrcorner H_{d-1} + \xi_{e} \lrcorner H_d)$. Therefore, we need to choose $\xi_{e} \in \Ann(H_{d-1}) \subseteq T$. If there exists a non-zero $\xi_{e} \in \Ann(H_{d-1}) \cap \Ann(H_{d})$ we can set $\xi_{e-1}=0$ and we are done.

Otherwise, let $h$ be the Hilbert function of $T/\Ann(H_{d-1})$. Then the $\CC$-vector space 
\[
\langle \xi_{e} \lrcorner H_d | \xi_{e}\in \Ann(H_{d-1}) \subseteq T \rangle
\]
has dimension $e+1-h(e)$. On the other hand the vector space $\langle \xi_{e-1} \lrcorner H_{d-1} | \xi_{e-1} \in T_{e-1} \rangle$ is of dimension $h(e-1)$. It is enough to show that these two vector subspaces of $T^*_{d-e}$ have a non-zero intersection. We claim that
\begin{equation}\label{eq:p^2-case}
e+1-h(e) + h(e-1) \geq d-e+2.
\end{equation}

By the definition of $e$ we have $d+1 \leq 2e$. We claim that $h(e-1)-h(e) \geq 0$. If $\Ann(H_{d-1})_{e-1} \neq 0$ then $h(e-1) - h(e) \geq 0$. On the other hand, if $\operatorname{Ann}(H_{d-1})_{e-1} = 0$ and $h(e) > h(e-1)$, then $\Ann(H_{d-1})_e = 0$. Thus, $h(d-1-e) = h(e) = e+1$ since $T/\Ann(H_{d-1})$ is Gorenstein. In particular, $d-1-e \geq e$. This contradicts $d+1 \leq 2e$. These remarks imply Equation \eqref{eq:p^2-case}.
\end{proof}

\begin{remark}
In \cite[page~37]{BGI09} in the paragraph above Remark 13, there is an example that suggests that there could exist a wild polynomial in $\sigma_{8}(\nu_6(\PP^2))$. Proposition \ref{p:wild_on_plane} shows that there is no such polynomial.
\end{remark}

It would be interesting to address the following questions related to Proposition \ref{p:wild_on_plane}.

\begin{problem}
Does there exist a homogeneous polynomial $F\in \CC[x_0,x_1,x_2]$ such that $\brr(F) \neq \srr(F)$? If it does, what is the smallest possible degree of such a polynomial?
\end{problem}

\subsection{Quartics in four variables of small border rank}
In this subsection we will prove the following proposition which shows that there are no wild homogeneous quartics $F$ in four variables of border rank at most $6$.

\begin{proposition}\label{p:wild_quartics_in_four_variables}
Let $S=\CC[\alpha_0,\alpha_1,\alpha_2, \alpha_3]$ be a polynomial ring with dual ring $S^*=\CC[x_0,x_1,x_2,x_3]$. Let $F\in S^*_4$ be non-zero. If the border rank of $F$ is at most $6$, then $\crr(F)=\srr(F) = \brr(F)$.
\end{proposition}
\begin{proof}
By \cite[Prop.~2.5]{BB14} we may assume that $\brr(F) = 6$. Since $\mathcal{H}ilb_6(\PP^3)$ is irreducible (see \cite[Thm.~1.1]{CEVV09}), it is enough to show that $\crr(F) \leq 6$. As in the previous subsection, using Theorem \ref{t:criterion_for_points_on_line} and Corollary~\ref{c:bound_on_cactus_rank} we may restrict our attention to the case that $F=x_0C_1+x_1C_2+D$ where $C_1,C_2 \in \CC[x_2,x_3]_3$ and $D\in \CC[x_2,x_3]_4$.

There is a polynomial $\theta \in \CC[\alpha_2,\alpha_3]_3$ such that $\theta \lrcorner C_1 = \theta \lrcorner C_2 = 0$. By a linear change of variables in $\CC[\alpha_2,\alpha_3]$ we may assume that $\theta$ is one of the following:
\begin{enumerate}\itemsep0em
\item $\theta  = \alpha_2^3$;
\item $\theta = \alpha_2^2\alpha_3$;
\item $\theta = \alpha_2\alpha_3(\alpha_2-\alpha_3)$.
\end{enumerate}

We study this case by case. We will further simplify $F$ by a linear change of variables and in each case we will find a homogeneous ideal $J\subseteq \Ann(F)$ whose initial ideal with respect to the lex order with $\alpha_2 > \alpha_3 > \alpha_1 > \alpha_0$ will be saturated and the Hilbert polynomial of the corresponding quotient algebra will be $6$. Thus, $\crr(F) \leq 6$ by Lemma \ref{l:initial_ideal_of_saturation} and Proposition \ref{p:cactus_apolarity}. In each case the given set of generators of $J$ will be a Gr\"obner basis. We may assume that $\Ann(F)_1 = 0$ by Proposition \ref{p:wild_on_plane} and \cite[\S3.1]{BB15}.

We start with case 1. Up to a linear change of variables in $S^*$ we have one of the cases:
\begin{enumerate}\itemsep0em
\item[1.A] $F=x_0(x_2^{2}x_3+ax_3^3)+x_1x_2x_3^2 + Q$ with $a\in \CC$ and $Q\in \CC[x_2,x_3]_4$;
\item[1.B] $F=x_0(x_2^2x_3+ax_2x_3^2)+x_1x_3^3+Q$ with $a\in \CC$ and $Q\in \CC[x_2,x_3]_4$;
\item[1.C] $F=x_0x_2x_3^2+x_1x_3^3+Q$ with $Q\in \CC[x_2,x_3]_4$.
\end{enumerate}

Let $\alpha_2^3 \lrcorner Q = Ax_2+Bx_3$. Corresponding to the above cases, the following ideals contained in $\Ann(F)$ show that the cactus rank of $F$ is at most $6$: 
\begin{enumerate}\itemsep0em
\item[1.A] $J=(\alpha_0^2, \alpha_0\alpha_1, \alpha_1^2, \alpha_0\alpha_2-\alpha_1\alpha_3, \alpha_1\alpha_2^2, \alpha_2^3-\frac{A}{2}\alpha_0\alpha_2\alpha_3-\frac{B}{2}\alpha_1\alpha_2\alpha_3)$;
\item[1.B] $J=(\alpha_0^2, \alpha_0\alpha_1, \alpha_1^2, \alpha_1\alpha_2, \alpha_0\alpha_2^2-\frac{1}{3}\alpha_1\alpha_3^2, \alpha_2^3-\frac{A}{2}\alpha_0\alpha_2\alpha_3+\frac{aA-B}{6}\alpha_1\alpha_3^2)$;
\item[1.C] $J=(\alpha_0^2, \alpha_0\alpha_1, \alpha_1^2, \alpha_1\alpha_2, \alpha_0\alpha_2^2, \alpha_2^3-\frac{A}{2}\alpha_0\alpha_3^2-\frac{B}{6}\alpha_1\alpha_3^2)$.
\end{enumerate}

Now we consider case 2, namely we assume that $\alpha_2^2\alpha_3 \lrcorner C_1 = \alpha_2^2\alpha_3 \lrcorner C_2 = 0$. Then, up to a linear change of variables in $S^*$ and excluding possibilities already considered in case 1, we have one of the cases:
\begin{enumerate}\itemsep0em
\item[2.A] $F=x_0(x_2^3+ax_3^3)+x_1(x_2x_3^2+bx_3^3) + Q$ with $a,b\in \CC$ and $Q\in \CC[x_2,x_3]_4$;
\item[2.B] $F=x_0(x_2^3+ax_2x_3^2)+x_1x_3^3 + Q$ with $a\in \CC$ and $Q\in \CC[x_2,x_3]_4$.
\end{enumerate}

Let $\alpha_2^2\alpha_3 \lrcorner Q = Ax_2 + Bx_3$. Corresponding to the above cases, the following ideals contained in $\Ann(F)$ show that the cactus rank of $G$ is at most $6$:
\begin{enumerate}\itemsep0em
\item[2.A] If $a\neq 0$ take $J=(\alpha_0^2, \alpha_0\alpha_1, \alpha_1^2, a\alpha_1\alpha_2-\frac{1}{3}\alpha_0\alpha_3, \alpha_0\alpha_2\alpha_3, \alpha_2^2\alpha_3-\frac{A}{6}\alpha_0\alpha_2^2 - \frac{B}{2}\alpha_1\alpha_2\alpha_3)$.

If $a=0$ take $J=(\alpha_0^2, \alpha_0\alpha_1, \alpha_1^2, \alpha_0\alpha_3, \alpha_1\alpha_2^2, \alpha_2^2\alpha_3-\frac{A}{6}\alpha_0\alpha_2^2 - \frac{B}{2}\alpha_1\alpha_2\alpha_3)$;
\item[2.B] Take $J=(\alpha_0^2, \alpha_0\alpha_1, \alpha_1^2, \alpha_1\alpha_2, \alpha_0\alpha_2\alpha_3 - \frac{a}{3}\alpha_1\alpha_3^2, \alpha_2^2\alpha_3-\frac{A}{6}\alpha_0\alpha_2^2 - \frac{B}{6}\alpha_1\alpha_3^2)$.
\end{enumerate}
Finally we consider case 3, that is we assume that $\alpha_2\alpha_3(\alpha_2-\alpha_3) \lrcorner C_1 = \alpha_2\alpha_3(\alpha_2-\alpha_3) \lrcorner C_2 = 0$. Then, up to a linear change of variables in $S^*$ and excluding possibilities considered in case 1, we have one of the cases:
\begin{enumerate}\itemsep0em
\item[3.A] $F=x_0(x_2^3+ax_3^3)+x_1(x_2^2x_3+x_2x_3^2 + bx_3^3) + Q$ with $a,b\in \CC$ and $Q\in \CC[x_2,x_3]_4$;
\item[3.B] $F=x_0(x_2^3 + ax_2^2x_3+ax_2x_3^2)+x_1x_3^3 + Q$ with $a\in \CC$ and $Q\in \CC[x_2,x_3]_4$.
\end{enumerate}

Let $(\alpha_2^2\alpha_3 - \alpha_2\alpha_3^2)\lrcorner Q = Ax_2 + Bx_3$. Corresponding to the above cases, the following ideals show that the cactus rank of $G$ is at most $6$:
\begin{enumerate}\itemsep0em
\item[3.A] If $a\neq 0$ take $J=(\alpha_0^2, \alpha_0\alpha_1, \alpha_1^2, a\alpha_1\alpha_2 + \frac{a}{3}\alpha_0\alpha_2 - a\alpha_1\alpha_3 + (b-\frac{1}{3})\alpha_0\alpha_3, \alpha_0\alpha_2\alpha_3, \alpha_2^2\alpha_3-\alpha_2\alpha_3^2-\frac{A}{6}\alpha_0\alpha_2^2 -\frac{B}{2}\alpha_1\alpha_2^2)$.

If $a=0$ take $J=(\alpha_0^2, \alpha_0\alpha_1, \alpha_1^2, \alpha_0\alpha_3, \alpha_1\alpha_2^2 + \frac{1}{3}\alpha_0\alpha_2^2 - \alpha_1\alpha_2\alpha_3, \alpha_2^2\alpha_3 - \alpha_2\alpha_3^2 - \frac{A}{6}\alpha_0\alpha_2^2 - \frac{B}{2}\alpha_1\alpha_2^2)$;
\item[3.B]
If $a=0$ then we are in case 2.B. Therefore, assume that $a\neq 0$. Take 
\[
J=(\alpha_0^2, \alpha_0\alpha_1, \alpha_1^2, \alpha_1\alpha_2, \alpha_0\alpha_2\alpha_3 -\alpha_0\alpha_3^2 - \frac{a}{3}\alpha_1\alpha_3^2, \alpha_2^2\alpha_3-\alpha_2\alpha_3^2-\frac{A}{6}\alpha_0\alpha_2^2 +\frac{aA-3B}{18}\alpha_1\alpha_3^2).
\]
\end{enumerate}
\end{proof}

\subsection{Wild quintic in four variables of border rank 7}
In Proposition \ref{p:wild_quartics_in_four_variables} we showed that there are no wild quartics in four variables of border rank $6$. In this subsection, we give an example of a wild quintic in four variables of border rank $7$.

\begin{proposition}
Let $S=\CC[\alpha_0, \ldots, \alpha_3]$ be a polynomial ring with graded dual ring $S^*=\CC[x_0,\ldots,x_3]$.
Let $F=x_0x_2^4+x_0x_2^3x_3+x_1x_2^2x_3^2+x_1x_3^4$. Then $\brr(F) = 7$ and $\crr(F) > 7$. Thus, $F$ is wild.
\end{proposition}
\begin{proof}
By \cite[Prop.~2.5]{BB14} it suffices to show that $\brr(F) \leq 7 < \crr(F)$. 

We have $(\alpha_0, \alpha_1)^2 \subseteq \Ann(F)$. Let $J=(\Ann(F)_{\leq 3}) + (\alpha_2^7)$. Then $[J]\in \Hilb_S^{h_{7,3}}$ and it follows from Theorem \ref{t:criterion_for_points_on_line} that there is an ideal $[J']\in \Slip_{7,3}$ such that $(J')_{\geq 5} = J_{\geq 5}$. In particular, $J'_5 = J_5 \subseteq \Ann(F)_5$ so $J'\subseteq \Ann(F)$. It follows from Proposition \ref{p:border_apolarity} that $\brr(F) \leq 7$.

Now we show that $\crr(F) > 7$. Otherwise, by Proposition \ref{p:cactus_apolarity} there exists a homogeneous, saturated ideal $K \subseteq \Ann(F)$ such that $S/K$ has Hilbert polynomial $7$. Since $H_{S/\Ann(F)}(a) = h_{7,3}(a)$ for $a\leq 3$, we have $K_{\leq 3} = \Ann(F)_{\leq 3}$. In particular, $(\alpha_0, \alpha_1) = \overline{(\Ann(F)_{\leq 3})} \subseteq K$. This contradicts the assumption that $K_1=\Ann(F)_1 = 0$.
\end{proof}

\subsection{Cubics in five variables of minimal border rank}
In this subsection we let $S=\CC[\alpha_0, \ldots, \alpha_4]$ be a polynomial ring and $S^* = \CC[x_0,\ldots, x_4]$ be the dual ring. Let $C\in S^*$ be a homogeneous polynomial of degree $3$. We say that $C$ is concise if $\Ann(C)_1 = 0$. It is known that  there exists a wild concise cubic in $S^*$ of border rank $5$ (see \cite[Thm.~1.3]{BB15}) and that a concise cubic in $S^*$ of border rank $5$ is wild if and only if its Hessian is zero (see \cite[Thm.~4.9]{HMV20}). Using Theorem~\ref{t:criterion_for_points_on_line} we obtain in a simple way that up to a linear change of variables, the cubic given in \cite[Thm.~1.3]{BB15}  is the unique wild, concise cubic in $S^*$ of border rank $5$:
\begin{equation}\label{eq:concise_cubic}
x_0x_3^2-x_1(x_3+x_4)^2+x_2x_4^2.
\end{equation}

By Proposition \ref{p:border_apolarity} there is an ideal $I\subseteq \Ann(C)$ such that $[I]\in \Slip_{5, 4}$. If the Hilbert function of $S/\overline{I}$ is not $h_{5, 1}$, then $I_3 = \overline{I}_3$, so $\crr(C) \leq 5$ by Corollary \ref{c:bound_on_cactus_rank}. Consequently, $C$ is not wild since $\crr(C) = \srr(C)$ (see \cite[Thm.~1.1]{CEVV09}). Therefore, we may assume that $\overline{I} = (\alpha_0,\alpha_1,\alpha_2, F(\alpha_3,\alpha_4))$ for some $F \in \CC[\alpha_3,\alpha_4]_5$. Since $[I]\in \Slip_{5, 4}$ it follows from Theorem \ref{t:criterion_for_points_on_line}  that $(\alpha_0,\alpha_1,\alpha_2)^2\cdot (\alpha_0,\alpha_1,\ldots, \alpha_4) \subseteq \Ann(C)$ and thus, $(\alpha_0,\alpha_1,\alpha_2)^2 \subseteq \Ann(C)$. Hence
\[
C=x_0Q_0+x_1Q_1+x_2Q_2 + C' \text { where } Q_0,Q_1,Q_2\in \CC[x_3,x_4]_2 \text{ and } C'\in\CC[x_3,x_4]_3.
\]
Moreover, $Q_0,Q_1,Q_2\in \CC[x_3, x_4]_2$ are linearly independent since $C$ is concise. Therefore, after a linear change of variables we may reduce $C$ to the form given in Equation \eqref{eq:concise_cubic}.

\begin{remark}
The annihilator ideal $\Ann(C)$ of a concise cubic $C$ has a minimal generator of degree $3$ (see \cite[Thm.~5.4]{buczyska2019apolarity} for a vast generalization). Therefore, the form given in Equation \eqref{eq:concise_cubic} should be compared to the form given in \cite[Thm.~4.5]{BBKT15}.
\end{remark} 

\section{A sufficient condition in the case of projective plane}\label{s:sufficient}
In this section we are concerned with $\PP^2$, so in order to simplify notation, we will write $S$ instead of $S[\PP^2]$. Let $T=\Bbbk[\alpha_1,\alpha_2] \subseteq \Bbbk[\alpha_0, \alpha_1, \alpha_2] = S$. Fix a positive integer $r$ and recall from Section \ref{s:introduction} that the irrelevant ideal $(\alpha_0,\alpha_1,\alpha_2)\subseteq S$ is denoted by $\mathfrak{m}$.

The main result is Theorem \ref{t:almost_general_implies_lip} which we establish in Subsection \ref{ss:proof_of_almost_general}. Proof of the theorem is based on Proposition \ref{p:tangent_space_of_extended_ideal} which is presented in Subsection \ref{ss:computing_dimension}.

\subsection{Computing dimension of tangent space}\label{ss:computing_dimension}
 Recall that if $M,N$ are $\ZZ$-graded $T$-modules  (respectively, $S$-modules) and $M$ is finitely generated then the Ext groups $\Ext^i_T(M,N)$ (respectively, $\Ext^i_S(M,N)$) are $\ZZ$-graded $T$-modules (respectively, $S$-modules) in a natural way (see \cite[Sec.~1.5]{BH98}). For a $\mathbb{Z}$-graded $T$-module $M$ and an integer $d$, $M(d)$ is the $\mathbb{Z}$-graded $T$-module given by $M(d)_e = M_{e+d}$ for all $e\in \mathbb{Z}$.

Consider a monomial ideal $I$ of $T$ such that $\dim_\Bbbk T/I = r$ for some $r\in \mathbb{Z}_{>0}$. The goal of this subsection is to compute $\dim_\Bbbk \Hom_S(I^{ex}, S/I^{ex})_0$ where $I^{ex}=S\cdot I\subseteq S$ is the extended ideal. 

\begin{lemma}\label{lem:homs_extends}
In the above notation we have
\[
\dim_\Bbbk \Hom_S(I^{ex}, S/I^{ex})_0 = \sum_{d\leq 0} \dim_\Bbbk \Hom_T(I, T/I)_d.
\]
\end{lemma}
\begin{proof}
There is a natural $\Bbbk$-linear map $\tau\colon S\to T$ given by $\tau(f)=f|_{\alpha_0=1}$. For an integer $a$ it induces a map $\xi_a\colon I^{ex}_a\to I_{\leq a}$. Furthermore, $\tau$ induces a map $\chi\colon S/I^{ex}\to T/I$ and for an integer $a$ we get an isomorphism of $\Bbbk$-vector spaces $\chi_a\colon (S/I^{ex})_a\to (T/I)_{\leq a}$.
Given $\varphi\in \Hom_S(I^{ex}, S/I^{ex})_0$ we define $\widetilde{\varphi} \in \Hom_T(I, T/I)_{\leq 0}$ by
\[
I\hookrightarrow I^{ex}\xrightarrow{\varphi} S/I^{ex} \xrightarrow{\chi} T/I.
\]
For $\psi\in \Hom_T(I, T/I)_{\leq 0}$ we define $\widehat{\psi}\in \Hom_S(I^{ex}, S/I^{ex})_0$ by $\widehat{\psi}(f_a) = \widehat{\psi}_a(f_a)$ where $a\in \ZZ$, $f_a\in I^{ex}_a$ and $\widehat{\psi}_a\in \Hom_\Bbbk(I^{ex}_a, (S/I^{ex})_a)$ is given by
\[
I^{ex}_a \xrightarrow{\xi_a} I_{\leq a} \xrightarrow{\psi} (T/I)_{\leq a} \xrightarrow{\chi_a^{-1}} (S/I^{ex})_a.
\]
The associations $\varphi\mapsto \widetilde{\varphi}$ and $\psi\mapsto \widehat{\psi}$ are inverse to each other. 
\end{proof}

We will use Lemma \ref{lem:homs_extends} in the case that $H_{S/I^{ex}} = h_{r, 2}$. Then $\dim_\Bbbk \Hom_S(I^{ex}, S/I^{ex})_0$ is the dimension of tangent space to $\Hilb_S^{h_{r,2}}$ at the point $[I^{ex}]$ (see \cite[Prop.~1.6]{HS02}).
Moreover, $\dim_\Bbbk \Hom_T(I,T/I) = 2r$ since this is the tangent space to the Hilbert scheme of 0-dimensional length $r$ closed subschemes of $\mathbb{A}^2$ and that scheme is smooth of dimension $2r$ (see \cite{Fog68}). Therefore,
\begin{equation*}
\dim_\Bbbk \Hom_S(I^{ex}, S/I^{ex})_0 = 2r - \sum_{d>0} \dim_\Bbbk \Hom_T(I, T/I)_d
\end{equation*}
and we may focus on calculating the latter sum.

Given a monomial ideal $I$ as above we have the associated staircase diagram (see \cite{miller2004combinatorial}) obtained as follows. 
For each pair of non-negative integers $(s,t)$ such that $\alpha_1^s\alpha_2^t \notin I$ we put a $1\times 1$ box with sides parallel to coordinate axis and $(s,t)$ as lower left corner of the box.
We will denote the diagram corresponding to $I$ by $\mathcal{D}_I$. The set of boxes of $\mathcal{D}_I$ (identified with the set of monomials outside $I$) will be denoted by $\Lambda_I$. We will use the canonical minimal free resolution of $T/I$ as given in \cite[Prop.~3.1]{miller2004combinatorial}. The set of minimal monomial generators of $I$ will be denoted by $M_I$ and the generating set of relations (or more precisely the set of their degrees when $T$ is considered with natural $\mathbb{Z}^2$-grading) used in that resolution will be denoted by $R_I$. Note that $\#\Lambda_I = \dim_\Bbbk T/I = r$  and  $\# R_I = \# M_I-1$. We will identify monomials of $T$ with lattice points in $\mathbb{Z}^2$. Given a point $\mathbf{u}=(s,t)$ in $\mathbb{Z}^2$ we will write $|\mathbf{u}|$ for $s+t$. 
We define three functions from integers to integers:
\begin{multicols}{2}
\begin{align*}
\lambda_I(a) =  &  \#\{\mathbf{u} \in \Lambda_I \mid |\mathbf{u}| = a\}, \\
\mu_I(a) = & \beta_{1, a} (T/I)   = \#\{\mathbf{u} \in M_I \mid |\mathbf{u}| = a\},\\
\rho_I(a) = & \beta_{2, a} (T/I) = \#\{\mathbf{u} \in R_I \mid |\mathbf{u}| = a\}.
\end{align*}
\columnbreak

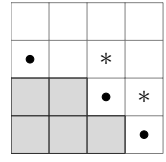
\begin{figure}[H]
\centering
  \begin{tikzpicture}[scale=0.50]
    \coordinate (Origin)   at (0,0);
    \coordinate (XAxisMin) at (-1,0);
    \coordinate (XAxisMax) at (5,0);
    \coordinate (YAxisMin) at (0,-1);
    \coordinate (YAxisMax) at (0,5);
  
    \clip (0,0) rectangle (4.1 cm,4cm);
	
	\draw[step=1cm,gray,very thin] (0,0) grid (4,4);
	   
	\draw [, black] (0,0) -- (3,0);
    \draw [, black] (0,0) -- (0,2);
    \draw [, black] (0,1) -- (3,1);
    \draw [, black] (0,2) -- (2,2);
    \draw [, black] (1,0) -- (1,2);
    \draw [, black] (2,0) -- (2,2);
    \draw [, black] (3,0) -- (3,1);
    
	\draw[color=black, fill=black]   (3.5,0.5) circle (.1);
	\draw[color=black, fill=black]   (2.5,1.5) circle (.1);
	\draw[color=black, fill=black]   (0.5,2.5) circle (.1);
	
	\node at (3.5, 1.5) {$\ast$};
	\node at (2.5, 2.5) {$\ast$};
	
	\filldraw[fill=gray, fill opacity=0.3, draw=gray, draw opacity = 0.3] (0,0)--(3,0)--(3,1)--(2,1)--(2,2)--(0,2);
  \end{tikzpicture}
\caption{Staircase diagram of the ideal $I=(\alpha_1^3, \alpha_1^2\alpha_2, \alpha_2^2).$}\label{f:example_of_a_diagram}
\end{figure}
\end{multicols}

\begin{example}
Figure \ref{f:example_of_a_diagram} presents the staircase diagram of $I=(\alpha_1^3, \alpha_1^2\alpha_2, \alpha_2^2)$. Filled boxes correspond to monomials outside $I$ (i.e. elements of $\Lambda_I$), dots correspond to elements of $M_I$ (i.e. minimal monomial generators of $I$) and asterisks correspond to elements of $R_I$ (i.e. minimal relations between those generators).
\end{example}

The goal of this subsection is the proof of the following proposition, which will play a key role in the proof of Theorem \ref{t:almost_general_implies_lip}.

\begin{proposition}\label{p:tangent_space_of_extended_ideal}
Given a monomial ideal $I$ in $T$ with $\dim_\Bbbk T/I = r$ for some $r\in \mathbb{Z}_{>0}$ we have
\begin{equation}\label{eqn:formula_for_positive_tangent_space}
\dim_\Bbbk \Hom_T(I, T/I)_{>0} = \sum_{\mathbf{u} \in M_{I}} \sum_{a > |\mathbf{u}|} \lambda_I(a) - \sum_{\mathbf{u} \in R_{I}} \sum_{a > |\mathbf{u}|} \lambda_I(a).
\end{equation}
\end{proposition}

The proof of Proposition \ref{p:tangent_space_of_extended_ideal} is based on the following observation.

\begin{lemma}\label{l:vanishing_of_exts}
Let $I$ be a monomial ideal in $T$ such that $T/I$ is a finite $\Bbbk$-vector space. Then:
\begin{enumerate}[label=(\roman*)]
\item[(i)] The natural map $T \to \Hom_T(I,T)$ given by $f\mapsto (g\mapsto fg)$ is an isomorphism of graded $T$-modules. 
\item[(ii)] $\Ext^1_T(I,T/I)_{>0}=0$.
\end{enumerate}
\end{lemma}
\begin{proof}
Since $\dim_\Bbbk T/I$ is finite, $I=(\alpha_1^{a_0}, \alpha_1^{a_1}\alpha_2^{b_1}, \ldots, \alpha_1^{a_{s-1}}\alpha_2^{b_{s-1}}, \alpha_2^{b_s})$ for some positive integers $a_0>a_1>\ldots > a_{s-1}$ and $b_1 < b_2 < \ldots < b_s$. Set $a_s=b_0=0$. 
\begin{enumerate}[label=(\roman*)]
\item[(i)] 
Let $\varphi\colon I\to T$ be a homomorphism of $T$-modules. It is enough to show that there exists an element $f\in T$ such that $\varphi(g) = fg$ for every $g\in I$.
Define $f_i=\varphi(\alpha_1^{a_i}\alpha_2^{b_i})$ for $i=0,\ldots, s$. Then for each $i\in \{1, \ldots, s\}$ we have relations of the form
\begin{equation}\label{eqn:lemma_vanishing_of_exts}
\alpha_2^{b_i-b_{i-1}}f_{i-1} = \varphi(\alpha_1^{a_{i-1}}\alpha_2^{b_i}) =  \alpha_1^{a_{i-1}-a_i}f_i.
\end{equation}
From Equation \eqref{eqn:lemma_vanishing_of_exts} for $i=s$ we deduce that $\alpha_1^{a_{s-1}}$ divides $f_{s-1}$. It follows by induction that $\alpha_1^{a_i}$ divides $f_i$ for all $i$. Thus, $f_0 = \alpha_1^{a_0}f$ for some $f\in T$. From Equation \eqref{eqn:lemma_vanishing_of_exts} we conclude that $f_i = \alpha_1^{a_i}\alpha_2^{b_i}f$ for each $i$.

\item[(ii)]
We start with showing that $\Ext^1_T(I,T)_{>0} = 0$. Consider the canonical minimal graded free resolution 
\[
\bigoplus_{i=1}^s T(-a_{i-1}-b_i) \to \bigoplus_{i=0}^s T(-a_i-b_i) \to I \to 0
\]
of $I$ (see \cite[Prop.~3.1]{miller2004combinatorial}).
Applying the functor  $\Hom_T(-, T)$ to the above resolution, we obtain for every integer $c$ a $\Bbbk$-linear map 
\[
\psi_c \colon \bigoplus_{i=0}^s T(a_i+b_i)_c \to \bigoplus_{i=1}^s T(a_{i-1}+b_i)_c.
\]
We claim that $\psi_c$ is surjective for every $c > 0$. Observe that $\ker \psi_c \cong \Hom_T(I,T)_c \cong T_c$ by part~(i). Therefore, the claim is a consequence of the  calculation
\[
\dim_\Bbbk \bigoplus_{i=0}^s T_{a_i+b_i+c} = (s+1)(c+1) + \sum_{i=1}^s (a_{i-1}+b_i) = \dim_\Bbbk \bigoplus_{i=1}^s T_{a_{i-1}+b_i+c} + \dim_\Bbbk T_c.
\]
Since $\psi_c$ is surjective for positive $c$, it follows that $\Ext_T^1(I, T)_{>0} = 0$. 

Now we prove that $\Ext^1_T(I,T/I)_{>0}=0$. Consider the following part of the long exact sequence of $\Ext$ groups obtained from the short exact sequence $0\to I \to T \to T/I \to 0$ by applying the functor $\Hom_T(I,-)$:
\[
\ldots \to \Ext^1_T(I,T)_{>0} \to \Ext^1_T(I,T/I)_{>0} \to \Ext^2_T(I,I)_{>0}\to \ldots.
\]
We have shown that $\Ext^1_T(I,T)_{>0} = 0$. Moreover, $\Ext^2_T(I,I)_{>0}$ since $I$ has projective dimension~$1$. It follows that $\Ext^1_T(I,T/I)_{>0}=0$.
\end{enumerate}
\end{proof}

\begin{proof}[Proof of Proposition \ref{p:tangent_space_of_extended_ideal}]
Consider the canonical minimal free resolution of $I$
\[
0\to \bigoplus_{a\in \ZZ}T(-a)^{\rho_{I}(a)} \to \bigoplus_{b\in \ZZ}T(-b)^{\mu_{I}(b)} \to I \to 0.
\]
Applying the functor $\Hom_T(-, T/I)_{>0}$ and using Lemma \ref{l:vanishing_of_exts}(ii) we get an exact sequence
\[
0 \to \Hom_T(I, T/I)_{>0} \to \bigoplus_{b\in \ZZ}\Hom_T(T(-b)^{\mu_{I}(b)}, T/I)_{>0} \to \bigoplus_{a\in \ZZ}\Hom_T(T(-a)^{\rho_{I}(a)}, T/I)_{>0} \to 0.
\]
This can be rewritten as
\[
0 \to \Hom_T(I, T/I)_{>0} \to \bigoplus_{b\in \ZZ}(T/I)^{^{\mu_{I}(b)}}_{>b} \to \bigoplus_{a\in \ZZ}(T/I)^{^{\rho_{I}(a)}}_{>a} \to 0.
\]
Thus,
\[
\dim_\Bbbk \Hom_T(I, T/I)_{>0} = \sum_{\mathbf{u}\in M_I} \sum_{c>|\mathbf{u}|} \dim_\Bbbk (T/I)_c  - \sum_{\mathbf{u}\in R_I} \sum_{c>|\mathbf{u}|} \dim_\Bbbk (T/I)_c.
\]
This is equivalent to Equation \eqref{eqn:formula_for_positive_tangent_space}.
\end{proof}
\subsection{Proof of Theorem \ref{t:almost_general_implies_lip}}\label{ss:proof_of_almost_general}
This subsection is devoted to proving Theorem \ref{t:almost_general_implies_lip}. We will start with a series of lemmas. The first one concerns Hilbert functions of saturated ideals defining zero-dimensional closed subschemes of $\mathbb{P}^2$.
\begin{lemma}\label{l:properties_of_hilbert_function_of_saturated_ideal}
Let $I\subseteq S$ be a homogeneous, saturated ideal such that the Hilbert polynomial of $S/I$ is constant. Then:
\begin{itemize}\itemsep0em
\item[(i)] $H_{S/I}(d+1)-H_{S/I}(d)\geq 0$ for all integers $d$;
\item[(ii)] If $H_{S/I}(d) = H_{S/I}(d-1)$ for a positive integer $d$, then $H_{S/I}(d+1)=H_{S/I}(d)$.
\end{itemize}
\end{lemma}
\begin{proof}
We may assume that $I\neq (1)$.
\begin{itemize}\itemsep0em
\item[(i)]
We shall prove that there is an element $f\in S_1$ that is a non-zero divisor on $S/I$. Let $\mathfrak{p}_1,\ldots, \mathfrak{p}_k$ be the associated primes of $S/I$. It is enough to show that $\bigcup_{i=1}^k (\mathfrak{p}_i)_1 \neq S_1$. Suppose that it does not hold. Since $\Bbbk$ is infinite we have $(\mathfrak{p}_i)_1 = S_1$ for some $i$ and therefore, $\mathfrak{m}$ is an associated prime of $S/I$. This gives a contradiction since $I$ is saturated.
Thus, the map $(S/I)_d \xrightarrow{\cdot f} (S/I)_{d+1}$ is injective for every $d$.

\item[(ii)] Suppose that $H_{S/I}(d)=H_{S/I}(d-1)$. Then $(S/I)_{d-1} \xrightarrow{\cdot f} (S/I)_d$ is an isomorphism of $\Bbbk$-vector spaces by the proof of (i). We claim that also  $(S/I)_{d} \xrightarrow{\cdot f} (S/I)_{d+1}$ is an isomorphism. It is injective since $f$ is a non-zero divisor on $S/I$. Let $g\in (S/I)_{d+1}$. Then $g=\alpha_0h_0+\alpha_1h_1+\alpha_2h_2$ for some $h_0,h_1,h_2\in (S/I)_{d}$. By assumptions, there are $k_0, k_1,k_2\in (S/I)_{d-1}$ such that $fk_0=h_0, fk_1=h_1$, and $fk_2=h_2$. It follows that $g=f(\alpha_0k_0+\alpha_1k_1+\alpha_2k_2)$.
\end{itemize}
\end{proof}

We will use the following observation to show that certain $\Ext$ groups vanish.

\begin{lemma}\label{l:exts_to_field}
Given a finitely generated graded $S$-module $M$ and an integer $e\in \mathbb{Z}$ we have 
\[
\dim_\Bbbk \Ext^i_S(M, \Bbbk)_e = \beta_{i,-e}(M).
\]
\end{lemma}
\begin{proof}
Apply the functor $\Hom_S(-, \Bbbk)$ to a minimal graded free resolution $P_\bullet$ of $M$. The Ext groups $\Ext^i_S(M, \Bbbk)$ can be computed as cohomology groups of the obtained complex. Since the $i$-th differential in $P_\bullet$ maps $P_i$ into $\mathfrak{m}P_{i-1}$, the differentials in the complex $\Hom_S(P_\bullet, \Bbbk)$  are zero.
Therefore, $\dim_{\Bbbk} \Ext_S^i(M, \Bbbk)_e = \dim_\Bbbk \Hom_S(P_i, \Bbbk)_e = \beta_{i, -e} (M)$.
\end{proof}

\begin{remark}\label{r:for_irreducible_slip_is_easy}
We have a natural map $\varphi_{r, \PP^2}\colon\Hilb_S^{h_{r, 2}} \to \mathcal{H}ilb_r(\mathbb{P}^2)$ given on closed points by $[I]\mapsto [\Proj S/I]$. This is a closed map that maps $\Sip_{r, 2}$ onto an open subset of the locus of reduced subschemes. Since $\mathcal{H}ilb_r(\mathbb{P}^2)$ is irreducible, it follows that $\varphi_{r, \PP^2}(\Slip_{r, 2}) = \mathcal{H}ilb_r(\mathbb{P}^2)$ set-theoretically. In particular, for every closed point $[I]\in \Hilb_S^{h_{r, 2}}$ there is a point $[I']\in \Slip_{r, 2}$ with $\overline{I} = \overline{I'}$.
As a special case, if $I$ is saturated and $[I]\in \Hilb_S^{h_{r, 2}}$ then $[I] \in \Slip_{r, 2}$.
\end{remark}

The next lemma enables us to consider only ideals $I$ for which the Hilbert function of $S/\overline{I}$ satisfies a more restrictive condition \eqref{eq:condition_star2} instead of \eqref{eq:condition_star}.

\begin{lemma}\label{l:reduction_to_star2}
Let $r$ be a positive integer and $[I]\in \Hilb_S^{h_{r, 2}}$. If the Hilbert function $f$ of $S/\overline{I}$ satisfies \eqref{eq:condition_star} for some integers $d,e$ then $[I]\in \Slip_{r,2}$ unless the following holds:
\begin{equation}\label{eq:condition_star2}
f \text{ satisfies } \eqref{eq:condition_star} \text{ and } \dim_\Bbbk S_{e-1} < d < r < \dim_\Bbbk S_e \tag{$\star\star$}.
\end{equation}
\end{lemma}
\begin{proof}
Assume that $[I]\in \Hilb_S^{h_{r, 2}}\setminus \Slip_{r, 2}$ and the Hilbert function of $S/\overline{I}$ satisfies \eqref{eq:condition_star} for some integers $d,e$.

Suppose that $d = r$. Then $f=h_{r, 2}$, so $[I]\in \Slip_{r, 2}$ by Remark \ref{r:for_irreducible_slip_is_easy}. Thus, by Lemma \ref{l:properties_of_hilbert_function_of_saturated_ideal}, we may assume that $d<r$.
Moreover, if $\dim_\Bbbk S_e \leq r$, then $[I]\in \Slip_{r, 2}$ since in that case $[I]$ is the unique closed point of the fiber over $\varphi_{r, \PP^2}([I])$ of the natural map $\varphi_{r, \PP^2}\colon \Hilb_S^{h_{r, 2}} \to \mathcal{H}ilb_r(\mathbb{P}^2)$ from Remark~\ref{r:for_irreducible_slip_is_easy}. Therefore, we may assume that $r < \dim_\Bbbk S_e$. We claim that it is enough to consider the case that $f(e-1)=h_{r, 2}(e-1)=\dim_\Bbbk S_{e-1}$. Indeed, otherwise $f(e-1)=h_{r, 2}(e-1)=r$ and this contradicts Lemma \ref{l:properties_of_hilbert_function_of_saturated_ideal} since $f(e)=d<r=f(e-1)$. Using Lemma \ref{l:properties_of_hilbert_function_of_saturated_ideal}(i) we obtain $\dim_\Bbbk S_{e-1} \leq d$ and moreover by Lemma \ref{l:properties_of_hilbert_function_of_saturated_ideal}(ii) this inequality is strict, since $r=f(e+1) > d$.
\end{proof}

For a fixed positive integer $r$, consider the set 
\[
\Omega_r = \{ f \colon \mathbb{Z} \to \mathbb{N} \mid f \text{ satisfies } \eqref{eq:condition_star2} \text{ and there exists } [I]\in\Hilb_S^{h_{r,2}} 
\text{ such that } S/\overline{I} \text{ has Hilbert function } f\}.
\]
The following lemma will enable us to reduce the proof of Theorem \ref{t:almost_general_implies_lip} to constructing, for every $f\in \Omega_r$, a point $[I_f]\in \Slip_{r, 2}$ that satisfies $H_{S/\overline{I_f}} = f$ and such that $[I_f]$ is smooth in $\Hilb_S^{h_{r, 2}}$.
 
\begin{lemma}\label{l:stratas_of_functions_are_irreducible}
Let $f\in \Omega_r$ for some positive integer $r$. Then the locus of closed points $[I]$ of $\Hilb_S^{h_{r, 2}}$ such that $H_{S/\overline{I}} = f$ is irreducible.
\end{lemma}
\begin{proof}
Denote this locus by $V$. Let $W\subseteq \mathcal{H}ilb_r(\PP^2)$ be the locally closed subset defined by closed points corresponding to closed subschemes of $\PP^2$ with Hilbert function $f$.
By definition, $V$ are the closed points of the preimage of $W$ under 
\[
\varphi_{r, \PP^2} \colon \Hilb_S^{h_{r, 2}}\to \mathcal{H}ilb_r(\mathbb{P}^2).
\]
The locus $W$ is irreducible by \cite{Got88}. Let $t=\dim_\Bbbk S_e$. Then for every $[I]\in V$ the fiber of $\varphi_{r, \PP^2} \colon V \to W$ over the closed point $[\Proj S/I]$ is $\operatorname{Gr}(t-r,\overline{I}_e)$. Thus, $V$ is irreducible as claimed. 
\end{proof}

Fix a positive integer $r$ and a function $f\in \Omega_r$, or equivalently a pair of integers $d,e$ corresponding to $f$. To simplify notation let $s:=\dim_\Bbbk S_{e-1}$ and we define $A_i = \alpha_1^i\alpha_2^{e-i}$ for $0\leq i \leq e$, $B_i = \alpha_1^i\alpha_2^{e+1-i}$ for $0\leq i \leq e+1$ and $C_i = \alpha_1^i\alpha_2^{e+2-i}$ for $0\leq i \leq e+2$ so as to make it is easier to distinguish between the generators of different degrees.
We define the ideals
\begin{equation}\label{eq:def_jf}
J_f=(A_e, A_{e-1}, \ldots, A_{d-s}, B_{d-s-1}, B_{d-s-2}, \ldots, B_{r-d}, C_{r-d-1}, C_{r-d-2}, \ldots, C_0)
\end{equation}
and
\begin{equation}
I_f  = (A_e, \ldots, A_{r-s},
A_{r-s-1}\alpha_2, A_{r-s-1}\alpha_0, \ldots, A_{d-s}\alpha_2, A_{d-s}\alpha_0,
  B_{d-s-1}, \ldots, B_{r-d}, C_{r-d-1}, \ldots, C_0).
\label{eq:def_if}
\end{equation}
Note that $A_i\alpha_2 = B_i$ but we have written $I_f$ in the form as above since it will be more convenient in the proof of the following lemma.

\begin{lemma}\label{l:I_f_in_V}
In the above notation, the Hilbert function of $S/I_f$ is $h_{r,2}$ and the Hilbert function of $S/\overline{I_f}$ is $f$. Moreover, $J_f$ is the saturation of $I_f$.
\end{lemma}
\begin{proof}
We start with showing that $J_f$ is a saturated ideal and that the Hilbert function of $S/J_f$ is $f$.
Indeed, $J_f$ is an extension of the ideal $\mathfrak{a}_f=J_f\cap T$ in $T$ so it is saturated with respect to $\alpha_0$ and thus, is saturated with respect to $\mathfrak{m}$. Moreover, $H_{S/J_f}(a) = \sum_{b=0}^a H_{T/\mathfrak{a}_f}(b)$ and the latter sum can be computed from the staircase diagram of $\mathfrak{a}_f$.

It follows from the generators of $I_f$ and $J_f$ displayed above, that the saturation of $I_f$ contains $J_f$. Therefore, $\overline{I_f} = J_f$ since $J_f$ is saturated. Because $S/J_f$ has Hilbert function $f$, the Hilbert function of $S/I_f$ is $h_{r,2}$.
\end{proof}

We will now find a saturated ideal $K_f$ such that the initial ideal of $K_f$ with respect to an appropriate monomial order is $I_f$.
Let
\begin{equation}\label{al:generators_of_kf}
\begin{split}
K_f = (&A_e, \ldots, A_{r-s},
 A_{r-s-1}\alpha_2, A_{r-s-1}\alpha_0 + B_{r-d-1}, \ldots, A_{d-s}\alpha_2, A_{d-s}\alpha_0 + B_0,\\
 &B_{d-s-1}, \ldots, B_{r-d}, C_{r-d-1}, \ldots, C_0).
\end{split}
\end{equation}

\begin{lemma}\label{l:kf_in_hilb}
In the above notation, the initial ideal of $K_f$ with respect to the lex order $>$ with $\alpha_0 >  \alpha_2 > \alpha_1$ is $I_f$. In particular, $S/K_f$ has Hilbert function $h_{r,2}$.
\end{lemma}
\begin{proof}
All $S$-polynomials of the generators displayed in Equation \eqref{al:generators_of_kf} are in the ideal $(C_{e+2}, \ldots, C_0)$. It follows that this set of generators satisfies the B\"uchberger criterion (see \cite[Thm.~6~in~Ch.~2~\S6]{CLO}) and is thus, a Gr\"obner basis. In particular, the initial ideal $\operatorname{in_<}(K_f)$ is $I_f$ so $S/K_f$ has Hilbert function $h_{r,2}$ by Lemma \ref{l:I_f_in_V}. 
\end{proof}

In order to prove that $K_f$ is saturated we will use the following result.

\begin{lemma}\label{l:initial_ideal_of_saturation}
Let $\mathfrak{a}$ be an ideal in $S$ and $<$ a monomial order. Then $
\operatorname{in}_{<}(\overline{\mathfrak{a}}) \subseteq \overline{\operatorname{in}_{<}(\mathfrak{a}) }.
$
\end{lemma}
\begin{proof}
Let $f\in \overline{\mathfrak{a}}$. Then there is an integer $l$ such that $(\alpha_0^lf, \alpha_1^lf, \alpha_2^lf) \subseteq \mathfrak{a}$. Therefore, 
\[
(\alpha_0^l\operatorname{in}_<(f),\alpha_1^l\operatorname{in}_<(f),\alpha_2^l\operatorname{in}_<(f) ) = (\operatorname{in}_<(\alpha_0^lf), \operatorname{in}_<(\alpha_1^lf), \operatorname{in}_<(\alpha_2^lf)) \subseteq \operatorname{in}_<(\mathfrak{a}).
\]
Thus, $\operatorname{in}_<(f)\in \overline{\operatorname{in}_{<}(\mathfrak{a})}$ which finishes the proof.
\end{proof}

\begin{lemma}\label{l:kf_is_saturated}
In the above notation, $K_f$ is a saturated ideal. In particular, $[I_f] \in \Slip_{r,2}$.
\end{lemma}
\begin{proof}
Let $>$ be the lex order with $\alpha_0 > \alpha_2 > \alpha_1$. From Lemmas \ref{l:I_f_in_V}, \ref{l:kf_in_hilb} and \ref{l:initial_ideal_of_saturation} we obtain
\begin{equation}\label{eq:inclusion_of_saturations}
\operatorname{in}_<(\overline{K_f})\subseteq \overline{\operatorname{in}_<(K_f)} = \overline{I_f} = J_f.
\end{equation}
Suppose that $\overline{K_f}\neq K_f$. Then $I_f = \operatorname{in}_<(K_f) \subsetneq \operatorname{in}_<(\overline{K_f}) \subseteq J_f$, where the containment in the middle is strict since the two ideals have different Hilbert functions by assumption. Since $I_f$, $J_f$ differ only in degree $e$, it follows that there is an element $g\in S_{e} \cap \overline{K_f}$ such that $\operatorname{in}_<(g)$ does not belong to the set of monomial generators of $I_f$ of degree $e$. From Equations \eqref{eq:def_jf} and \eqref{eq:inclusion_of_saturations} we get that $g=\sum_{i=d-s}^{e} a_i A_i$ for some $a_i\in \Bbbk$. 
We assumed that $\operatorname{in}_<(g)\notin I_f$. Thus, by Equation \eqref{eq:def_if}, we have $a_i \neq 0$ for some $i\in \{d-s, \ldots, r-s-1\}$.
Furthermore, we may assume that $a_i = 0$ for $i=r-s, \ldots, e$ by Equation \eqref{al:generators_of_kf}.
Multiplying $g$ by $\alpha_0^2$ and using the generators of $K_f$ given in Equation \eqref{al:generators_of_kf} we obtain
\[
g':=-\alpha_0^2g + \sum_{i={d-s}}^{r-s-1}a_i\alpha_0(A_i\alpha_0+B_{i+s-d}) = \sum_{i={d-s}}^{r-s-1} a_i B_{i+s-d}\alpha_0  \in \overline{K_f}.
\]
We claim that it is not possible. By Equation \eqref{eq:inclusion_of_saturations}, it is enough to show that no monomial of the form $B_{j}\alpha_0$ for $j\in \{0,\ldots, r-d-1\}$ is in $J_f$. This is clear since monomials of degree $e+2$ in $J_f$ that are divisible by $\alpha_0$ are also divisible by $\alpha_1^{r-d}$.

Since $K_f$ is saturated, it follows from Remark \ref{r:for_irreducible_slip_is_easy} that $[K_f] \in \Slip_{r,2}$. Therefore, $[I_f]\in \Slip_{r,2}$ by Lemma \ref{l:kf_in_hilb}.
\end{proof}

Finally, we will show that $[I_f]$ is a smooth point of $\Hilb_S^{h_{r, 2}}$. It is enough to show that 
\[
\dim_\Bbbk T_{[I_f]}\Hilb_S^{h_{r,2}} \leq \dim_\Bbbk \Slip_{r,2} = 2r.
\]
Lemmas \ref{l:ts_estimation_term_1} - \ref{l:ts_dim_is_2r} are devoted to this calculation. In what follows, we will consider $\Bbbk$ as a $\ZZ$-graded $S$-module via isomorphism $\Bbbk \cong S/\mathfrak{m}$. Then $J_f/I_f \cong \Bbbk(-e)^{\oplus (r-d)}$.

\begin{lemma}\label{l:ts_estimation_term_1}
In the above notation, $\dim_\Bbbk \Hom_S(J_f/I_f, S/I_f)_0 = (r-d)^2$.
\end{lemma}
\begin{proof}
Since $J_f/I_f \cong \Bbbk(-e)^{\oplus (r-d)}$ we have
\[
\begin{split}
&\dim_\Bbbk \Hom_S(J_f/I_f, J_f/I_f)_0 =(r-d)^2 \dim_\Bbbk \Hom_S(\Bbbk(-e), \Bbbk(-e))_0 \\
& =  (r-d)^2 \dim_\Bbbk \Hom_S(\Bbbk, \Bbbk)_0 =  (r-d)^2.
\end{split}
\]
Since $(J_f\colon \alpha_0) = J_f$ and $\alpha_0\cdot J_f/I_f =0$ it follows that $\Hom_S(J_f/I_f, S/J_f)_0 = 0$. Therefore, 
\[
\dim_\Bbbk \Hom_S(J_f/I_f, S/I_f)_0 = \dim_\Bbbk \Hom_S(J_f/I_f, J_f/I_f)_0 = (r-d)^2
\]
from the long exact sequence obtained by applying $\Hom_S(J_f/I_f, -)_0$ to the short exact sequence
\begin{equation}\label{eq:main_exact_sequence}
0\to J_f/I_f \to S/I_f \to S/J_f \to 0.
\end{equation}
\end{proof}

\begin{lemma}\label{l:ts_estimation_term_2}
In the above notation, $\dim_\Bbbk \Ext^1_S(J_f/I_f, S/I_f)_0 = (r-d)^2$.
\end{lemma}
\begin{proof}
We claim that $\dim_\Bbbk \Ext^i_S(J_f/I_f, J_f/I_f)_0 = 0$ for $i=1,2$. It is enough to show that 
\[
\dim_\Bbbk \Ext^i_S (\Bbbk, \Bbbk)_0 = 0
\]
since $\Ext^i_S(J_f/I_f, J_f/I_f)_0  = (\Ext^i_S(\Bbbk, \Bbbk)_0)^{\oplus (r-d)^2}$. Therefore, the claim follows from Lemma \ref{l:exts_to_field}.

Applying the functor $\Hom_S(J_f/I_f, -)_0$ to the short exact sequence \eqref{eq:main_exact_sequence} we obtain 
\[
\dim_\Bbbk \Ext^1_S(J_f/I_f, S/I_f)_0 = \dim_\Bbbk \Ext^1_S(J_f/I_f, S/J_f)_0.
\]
We have $\Ext^1_S(J_f/I_f, S/J_f)_0 \cong ( \Ext^1_S(\Bbbk, S/J_f)_{e})^{\oplus {(r-d)}}$
so  it is enough to compute the dimension of \\
$\Ext^1_S(\Bbbk, S/J_f)_{e}$ as a $\Bbbk$-vector space.

Applying the functor $\Hom_S(-, S/J_f)_e$ to the Koszul resolution of $\Bbbk$ we obtain the following complex:
\[
(S/J_f)_e \xrightarrow{{\begin{bmatrix}\alpha_0 \\ \alpha_1 \\ \alpha_2 \end{bmatrix}}
} (S/J_f)_{e+1}^{\oplus 3} \xrightarrow{{\begin{bmatrix}-\alpha_1  & \alpha_0 & 0 \\-\alpha_2 & 0 & \alpha_0 \\0 & -\alpha_2 & \alpha_1\end{bmatrix}}} (S/J_f)_{e+2}^{\oplus 3} \xrightarrow{{\begin{bmatrix}\alpha_2 & -\alpha_1 & \alpha_0\end{bmatrix}}} (S/J_f)_{e+3}.
\]
We need to show that the cohomology at $(S/J_f)_{e+1}^{\oplus 3}$ is an $(r-d)$-dimensional $\Bbbk$-vector space. We will denote the map $(S/J_f)_e \to (S/J_f)_{e+1}^{\oplus 3}$ by $d_0$ and the map $(S/J_f)_{e+1}^{\oplus 3} \to (S/J_f)_{e+2}^{\oplus 3}$ by $d_1$.
Let $h_1,h_2,h_3\in S_{e+1}$ be such that $d_1(\overline{h_1},\overline{h_2},\overline{h_3}) = 0$, where $\overline{h_i}$ is the class of $h_i$ in the quotient ring $S/J_f$. Let $h_1=\alpha_0h_1'(\alpha_0,\alpha_1,\alpha_2)+h_1''(\alpha_1,\alpha_2)$. Then $(-\alpha_1 h_1'' + \alpha_0(h_2-\alpha_1h_1'), - \alpha_2h_1'' + \alpha_0(h_3-\alpha_2h_1'), \alpha_1h_3 - \alpha_2h_2) \subseteq J_f$.
The ideal $J_f$ is monomial. Therefore, since $(J_f\colon \alpha_0) = J_f$ and $h_1''$ does not depend on $\alpha_0$, we get
$(h_2-\alpha_1h_1',h_3-\alpha_2h_1') \subseteq J_f$.
Thus, $(\overline{h_1}, \overline{h_2}, \overline{h_3}) = (\overline{h_1''},0,0) + d_0(\overline{h_1'})$. 
We claim that $(\overline{h_1''},0,0) \in \ker d_1$ for every degree $e+1$ homogeneous polynomial $h_1''\in \Bbbk[\alpha_1,\alpha_2]$. Indeed, we have $\alpha_1h_1'', \alpha_2h_1''\in J_f$ as $(J_f)_{e+2} \cap \Bbbk[\alpha_1,\alpha_2] =\Bbbk[\alpha_1,\alpha_2]_{e+2}$.
It follows that
\[
\dim_\Bbbk \Ext^1_S(\Bbbk, S/J_f)_e = \dim_\Bbbk \Bbbk[\alpha_1,\alpha_2]_{e+1}/(J_f\cap \Bbbk[\alpha_1, \alpha_2])_{e+1} = (r-d).
\]
\end{proof}
\begin{lemma}\label{l:ts_estimation_term_3}
In the above notation, $\dim_\Bbbk \Hom_S(J_f, S/I_f)_0 = 2r$.
\end{lemma}
\begin{proof}
Observe that $\Ext^1_S(J_f, J_f/I_f)_0 \cong (\Ext^1_S(J_f, \Bbbk)_{-e})^{\oplus (r-d)} = 0$ is zero by Lemma \ref{l:exts_to_field}. Indeed $(J_f)_{e-1} = 0$ so $\beta_{1, e}(J_f) = \beta_{2, e}(S/J_f) = 0$.

Therefore, applying the functor $\Hom_S(J_f, -)_0$ to the short exact sequence \eqref{eq:main_exact_sequence} we obtain a short exact sequence
\begin{equation}\label{eq:ts_est_t3A}
0\to \Hom_S(J_f, J_f/I_f)_0 \to \Hom_S(J_f, S/I_f)_0 \to \Hom_{S}(J_f, S/J_f)_0 \to 0.
\end{equation}
Since the minimal degree of generator of $J_f$ is $e$ we have
\begin{equation}\label{eq:ts_est_t3B}
\dim_\Bbbk \Hom_S(J_f, J_f/I_f)_0 = (r-d) \dim_\Bbbk \Hom_S(J_f, \Bbbk(-e))_0
= (r-d) \dim_\Bbbk (J_f)_e = (r-d)(\dim_\Bbbk S_e - d).
\end{equation}
Finally, let $\mathfrak{a}_f= J_f\cap T.$ Then $J_f = \mathfrak{a}_fS$, which implies by Lemma \ref{lem:homs_extends} that
\[
\dim_\Bbbk \Hom_S(J_f, S/J_f)_0 = \sum_{i\leq 0} \dim_\Bbbk \Hom_T(\mathfrak{a}_f, T/\mathfrak{a}_f)_i.
\]
Since $\Spec (T/\mathfrak{a}_f)$ corresponds to a point of the Hilbert scheme $\mathcal{H}ilb_{r}(\mathbb{A}^2)$ which is smooth and $2r$-dimensional we get 
\begin{equation}\label{eq:ts_est_t3C}
\sum_{i\leq 0} \dim_\Bbbk \Hom_T(\mathfrak{a}_f, T/\mathfrak{a}_f)_i = 2r - \sum_{i > 0} \dim_\Bbbk \Hom_T(\mathfrak{a}_f, T/\mathfrak{a}_f)_i  
= 2r - (r-d)(\dim_\Bbbk S_e - d)
\end{equation}
where the last equality follows from Proposition \ref{p:tangent_space_of_extended_ideal}. The exact sequence \eqref{eq:ts_est_t3A}
and Equations \eqref{eq:ts_est_t3B}, \eqref{eq:ts_est_t3C} imply that $\dim_\Bbbk \Hom_S(J_f, S/I_f)_0 = 2r$.
\end{proof}

\begin{lemma}\label{l:ts_dim_is_2r}
In the above notation, $\dim_\Bbbk \Hom_S(I_f, S/I_f)_0 \leq 2r$.
\end{lemma}
\begin{proof}
From the long exact sequence obtained by applying the functor $\Hom_S(-, S/I_f)_0$ to the short exact sequence 
$0\to I_f \to J_f \to J_f/I_f \to 0$
we get
\[
\dim_\Bbbk \Hom_S(I_f, S/I_f)_0 \leq 
\dim_\Bbbk\Hom_S(J_f, S/I_f)_0 + \dim_\Bbbk \Ext^1_S(J_f/I_f, S/I_f)_0 - \dim_\Bbbk \Hom_S(J_f/I_f, S/I_f)_0.
\]
Using Lemmas \ref{l:ts_estimation_term_1}, \ref{l:ts_estimation_term_2} and \ref{l:ts_estimation_term_3} we conclude that $\dim_\Bbbk \Hom_S(I_f, S/I_f)_0 \leq 2r$.
\end{proof}

We summarize the above reasoning to obtain a proof of Theorem \ref{t:almost_general_implies_lip}.
\begin{proof}[Proof of Theorem \ref{t:almost_general_implies_lip}]
Let $f$ be the Hilbert function of $S/\overline{I_0}$.
By Lemma \ref{l:reduction_to_star2} we may assume that $f$ satisfies condition \eqref{eq:condition_star2}, i.e. that $f\in \Omega_r$. Let $V$ be the locus of those closed points $[I]$ of $\Hilb_S^{h_{r,2}}$ for which $S/\overline{I}$ has Hilbert function $f$. We shall show that $V\subseteq \Slip_{r,2}$.
Locus $V$ is irreducible by Lemma \ref{l:stratas_of_functions_are_irreducible}.
Therefore, by Lemma \ref{l:components_of_intersection} it is enough to find a point $[I_1]\in \Slip_{r,2}\cap V$ such that $\dim_\Bbbk T_{[I_1]} \Hilb_S^{h_{r,2}} = 2r$.

We claim that we may take $I_1 = I_f$ as defined by Equation \eqref{eq:def_if}. We have $[I_f]\in V\cap \Slip_{r,2}$ by Lemmas~\ref{l:I_f_in_V} and \ref{l:kf_is_saturated}. Moreover, $[I_f]$ is a smooth point of $\Hilb_S^{h_{r,2}}$ by Lemma \ref{l:ts_dim_is_2r}.
\end{proof}

\subsection{Theorem \ref{t:almost_general_implies_lip} does not hold in general for projective space}
For fixed positive integers $r,n$, condition $\eqref{eq:condition_star}$ can be generalized as follows:
\begin{equation}\label{eq:condition_star_general}
\exists e,d\in \mathbb{Z}_{>0} \text{ such that } f(a) = \begin{cases}
h_{r, n}(a) & \text{ if } a\neq e\\
d & \text{ if }a=e.
\end{cases}
\tag{$\star\star\star$}
\end{equation}

Since for $n\geq 3$ and $r$ large enough, the Hilbert scheme $\mathcal{H}ilb_r(\mathbb{P}^n)$ is reducible (see \cite{Iar72}), it cannot be expected that a naive analogue of Theorem \ref{t:almost_general_implies_lip} holds in $\mathbb{P}^n$. The following remark gives a counterexample.

\begin{remark}\label{r:generalization_1}
There are non-smoothable closed subschemes of $\mathbb{P}^6$ with Hilbert function 
\[
(1,7,13,14, 14 \ldots)
\]
(see \cite{CJN15}, \cite{Jel16}). Given any such subscheme $R$ and its homogeneous ideal $I=I(R)$, we can choose any $14$-dimensional subspace $V$ of $I_2$ and construct an ideal $J=V\oplus I_{\geq 3}$. Then $[J]\in \Hilb_{S[\PP^6]}^{h_{14, 6}}$, Hilbert function of $S[\PP^6]/\overline{J}$ satisfies \eqref{eq:condition_star_general} (with $r=14, n=6$) but $[J]\notin \Slip_{14, 6}$.
\end{remark}

However, one could expect that existence of more irreducible components of $\mathcal{H}ilb_{r}(\mathbb{P}^n)$ is the only obstacle. The following example shows that the Hilbert function of $S/\overline{I}$ may satisfy condition \eqref{eq:condition_star_general} for $[I] \in \Hilb_{S[\PP^n]}^{h_{r,n}}$ that is not in the closure of the locus of points corresponding to saturated ideals.

\begin{remark}\label{r:generalization_2}
Consider the ideal
 \[
I=(\alpha_0^2\alpha_1,\alpha_0\alpha_1^2,\alpha_0\alpha_2,\alpha_0\alpha_3,\alpha_1\alpha_2,\alpha_1\alpha_3,\alpha_2^4) \in S[\PP^3] = \Bbbk[\alpha_0,\ldots, \alpha_3].
\]
Then $[I] \in \Hilb_{S[\PP^3]}^{h_{6, 3}}$ and the Hilbert function of $S/\overline{I}$ satisfies \eqref{eq:condition_star_general} with $r=6, n=3$. Suppose that $[I]$ is in the closure of the locus of points corresponding to saturated ideals. Since, $\mathcal{H}ilb_{6}(\PP^3)$ is irreducible (see \cite[Thm.~1.1]{CEVV09}) it follows that $[I]\in \Slip_{6, 3}$. This contradicts Theorem \ref{t:hilbert_function_of_square_criterion} since $H_{S/I^2}(6)= 23 < r(n+1)=24$.
\end{remark}

\subsection{Example of a singular point in the interior of Slip for projective plane}
Since $\mathcal{H}ilb_{r}(\PP^2)$ is smooth and $\Slip_{r, 2}$ is related to $
\mathcal{H}ilb_{r}(\PP^2)$ by the natural morphism 
\[
\varphi_{r, \PP^2} \colon \Hilb_{S[\PP^2]}^{h_{r, 2}} \to \mathcal{H}ilb_{r}(\PP^2),
\]
it could be expected that the only singular points of $\Hilb_{S[\PP^2]}^{h_{r, 2}}$ in $\Slip_{r, 2}$ are the points that lie in another irreducible component. We will show that it is not true.

We start with introducing some notation. Let $\Theta_{8, \PP^2}$ be the set of all functions $f\colon \mathbb{Z} \to \mathbb{Z}$ such that $f$ is the Hilbert function of a saturated homogeneous ideal of $S$ defining a zero-dimensional closed subscheme of $\PP^2$ of length $8$. Given $f\in \Theta_{8, \PP^2}$, let $U_f$ be the locally closed subscheme of $\mathcal{H}ilb_{r}(\PP^2)$ defined by those closed points that correspond to subschemes of $\PP^2$ with Hilbert function $f$. These sets with varying $f\in \Theta_{8, \PP^2}$ form a stratification of $\mathcal{H}ilb_{r}(\PP^2)$ by locally closed irreducible subsets (see \cite{Got88}). Let $V_f$ be the set-theoretic inverse image of $U_f$ under $\varphi_{r, \PP^2}$. In particular, $\Slip_{8, 2} = \overline{V_{h_{8, 2}}}$. Also, we say that $f \leq g$ for $f,g\colon \mathbb{Z} \to \mathbb{Z}$ if for every $a\in \mathbb{Z}$ we have $f(a) \leq g(a)$. This gives a partial order on $\Theta_{8, \PP^2}$.

Let $f_1,f_2 \colon \mathbb{Z} \to \mathbb{Z}$ be given by 
\[
{f_1(a) = \begin{cases}
\dim_\Bbbk S_a & \text{ for } a<3 \\
7 & \text{ for } a=3 \\
8 & \text{ for } a>3
\end{cases}
}
\quad \text{ and } \quad
   {f_2(a) = \begin{cases}
\dim_\Bbbk S_a & \text{ for } a<2 \\
5 & \text{ for } a=2 \\
7 & \text{ for } a=3 \\
8 & \text{ for } a>3
\end{cases}}.
\]

Let $I_1 = (\alpha_1^2\alpha_2,\alpha_0\alpha_1^2+\alpha_1\alpha_2^2,\alpha_1^4, \alpha_1\alpha_2^3, \alpha_2^5)$. Then the Hilbert function of $S/\overline{I_1}$ 
is $f_1$ so $[I_1] \in \Slip_{8, 2}$ by Theorem \ref{t:almost_general_implies_lip}. Hence also its initial ideal with respect to lex order ($\alpha_0>\alpha_1>\alpha_2$), i.e. 
\[
I_2 = (\alpha_0\alpha_1^2, \alpha_1^2\alpha_2, \alpha_1^4, \alpha_1\alpha_2^3, \alpha_2^5),
\]
is in $\Slip_{8, 2}$. Since $\dim_\Bbbk T_{[I_2]} \Hilb_{S[\PP^2]}^{h_{8,2}} = 16 = \dim \Slip_{8, 2}$, it follows that every irreducible closed subset of $\Hilb_{S[\PP^2]}^{h_{8, 2}}$ passing through $[I_2]$ is contained in $\Slip_{8, 2}$. In particular
\[
I_3 =(\alpha_1^3, \alpha_1^2\alpha_2, \alpha_1^2\alpha_0^2, \alpha_1\alpha_2^3, \alpha_2^5)
\]
belongs to $\Slip_{8, 2}$ since $[I_2], [I_3]$ lie in $V_{f_2}$ which we claim is irreducible. Indeed, restriction of $\varphi_{8, \PP^2}$ to the subschemes $U_{f_2}, V_{f_2}$ with reduced structures gives a proper, surjective morphism $V_{f_2}\to U_{f_2}$ to an (irreducible) variety such that over every closed point of $U_{f_2}$, the fiber is irreducible and of dimension $2$. Therefore, $V_{f_2}$ is irreducible.

We have $\dim_\Bbbk T_{[I_3]} \Hilb_{S[\PP^2]}^{h_{8, 2}} = 17 > \dim \Slip_{8, 2}$. Let $Z$ be an irreducible component of $\Hilb_{S[\PP^2]}^{h_{8, 2}}$
containing $[I_3]$. We will show that $Z = \Slip_{8, 2}$. We claim that $V_{f_2}$ is open in $\Hilb_{S[\PP^2]}^{h_{8, 2}} \setminus (V_{h_{8, 2}} \cup V_{f_1})$. It suffices to show that $U_{f_2}$ is open in $\mathcal{H}ilb_8(\PP^2)\setminus (U_{h_{8,2}}\cup U_{f_1})$. Since $f_2$ is the greatest element of $\Theta_{8, \PP^2}\setminus \{h_{8, 2}, f_1\}$ it follows that $U_{f_2}$ is defined by the closed points of $\mathcal{H}ilb_8(\PP^2)\setminus (U_{h_{8,2}}\cup U_{f_1})$ which correspond to subschemes with Hilbert function at least $f_2$. The set $\Theta_{8, \PP^2}$ is finite by Lemma \ref{l:properties_of_hilbert_function_of_saturated_ideal}. Therefore, $U_{f_2}$ is open in $\mathcal{H}ilb_8(\PP^2)\setminus (U_{h_{8,2}}\cup U_{f_1})$ by the semi-continuity theorem \cite[Thm.~III.12.8]{H77} applied to the Serre twists of the ideal sheaf of the universal subscheme of $\mathcal{H}ilb_r(\PP^n)\times \PP^n$. This concludes the proof of the claim.
Let $\eta$ be the generic point of $Z$. If $\eta \notin V_{h_{8, 2}} \cup V_{f_1} \cup V_{f_2}$ then $[I_3] \notin \overline{\{\eta\}}$ since $V_{f_2}$ is open in $\Hilb_{S[\PP^2]}^{h_{8, 2}} \setminus (V_{h_{8, 2}} \cup V_{f_1})$. Therefore, $Z = \overline{\{\eta\}} \subseteq \Slip_{8, 2}$ since by what was shown above, $V_{f_1} \cup V_{f_2} \subseteq \Slip_{8, 2}$.

\section*{Acknowledgements}

I would like to thank Jaros{\l{}}aw Buczy{\'n}ski for introduction to this subject, many discussions
and constant support. I am grateful to Joachim Jelisiejew for his suggestions concerning the proofs of Theorem \ref{t:hilbert_function_of_square_criterion} and Proposition \ref{p:tangent_space_of_extended_ideal}.

This work was supported by National Science Center, Poland, projects number 
2017/26/E/ST1/00231 and 2019/33/N/ST1/00858.

Computations were done in Macaulay2 \cite{M2}.

\bibliographystyle{abbrv}

\end{document}